\documentclass[12pt,leqno]{amsart}
\usepackage{amssymb}
\usepackage{amsmath,amssymb}
\newtheorem{theorem}{Theorem}[section]
\newtheorem{corollary}{Corollary}[section]

\newtheorem{proposition}{Proposition}[section]

\begin{document}
\theoremstyle{plain}
\newtheorem{MainThm}{Theorem}
\newtheorem{thm}{Theorem}[section]
\newtheorem{clry}[thm]{Corollary}
\newtheorem{prop}[thm]{Proposition}
\newtheorem{lem}[thm]{Lemma}
\newtheorem{deft}[thm]{Definition}
\newtheorem{hyp}{Assumption}
\newtheorem*{ThmLeU}{Theorem (J.~Lee, G.~Uhlmann)}

\theoremstyle{definition}
\newtheorem{rem}[thm]{Remark}
\newtheorem*{acknow}{Acknowledgments}
\numberwithin{equation}{section}
\newcommand{\eps}{{\varphi}repsilon}
\renewcommand{\d}{\partial}
\newcommand{\re}{\mathop{\rm Re} }
\newcommand{\im}{\mathop{\rm Im}}
\newcommand{\R}{\mathbf{R}}
\newcommand{\C}{\mathbf{C}}
\newcommand{\N}{\mathbf{N}}
\newcommand{\D}{C^{\infty}_0}
\renewcommand{\O}{\mathcal{O}}
\newcommand{\dbar}{\overline{\d}}
\newcommand{\supp}{\mathop{\rm supp}}
\newcommand{\abs}[1]{\lvert #1 \rvert}
\newcommand{\csubset}{\Subset}
\newcommand{\detg}{\lvert g \rvert}
\title[Partial data in two dimensions]{Partial data for the Calder\'on problem
in two dimensions}

\author[O. Imanuvilov]{Oleg Yu. Imanuvilov}
\address{Department of Mathematics, Colorado State
University, 101 Weber Building, Fort Collins CO, 80523 USA\\
e-mail: oleg@math.colostate.edu}
\thanks{First author partly supported by NSF grant DMS 0808130}

\author[G. Uhlmann]{Gunther Uhlmann}
\address{Department of Mathematics, University of Washington, Seattle,
WA 98195 USA\\
e-mail: gunther@math.washington.edu}
\thanks{Second author partly supported by NSF and a Walker Family Endowed
Professorship}

\author[M. Yamamoto]{Masahiro Yamamoto}
\address{Department of Mathematics, University of Tokyo, Komaba, Meguro,
Tokyo 153, Japan \\e-mail:  myama@ms.u-tokyo.ac.jp}

\begin{abstract}
We show in two dimensions that measuring Dirichlet
data for the conductivity equation on an open subset of the boundary
and, roughly speaking, Neumann data in slightly larger set than the
complement uniquely determines the conductivity on a simply
connected domain. The proof is reduced to show a similar result for
the Schr\"odinger equation. Using Carleman estimates with degenerate
weights we construct appropriate complex geometrical optics
solutions to prove the results.
\end{abstract}
\maketitle
\setcounter{tocdepth}{1}
\setcounter{secnumdepth}{2}

\section{\bf Introduction}

 This paper is concerned with the Electrical Impedance Tomography (EIT) inverse problem. The EIT inverse problem  consists in
determining the electrical conductivity of a body by making voltage and
current measurements at the boundary of the body. Substantial progress
has
been made on this problem since Calder{\'o}n's pioneer contribution
\cite{C}. This inverse problem is known also as the Calder{\'o}n problem.
This problem can be reduced to studying the Dirichlet-to-Neumann (DN) 
map associated to the Schr{\"o}dinger equation. A key ingredient in 
several of
the results is the construction of complex geometrical optics solutions
for the Schr{\"o}dinger equation (see \cite{U} for a recent survey). 
By this method in
dimensions $n\ge 3$ for the conductivity equation, the first global
uniqueness result for $C^2$ conductivities was proven in \cite{SU} and
the regularity was improved to having
$3/2$ derivatives in \cite{BT} and \cite{PPU}.
More singular conormal conductivities were considered in \cite{GLU}.
The uniqueness results were proven also for the Schr\"odinger equation.

In two dimensions the first global uniqueness result for the
Calder\'on problem with full data is in \cite{N} for
$C^2$-conductivities, and this was improved to
Lipschitz conductivities in \cite{B-U} and for merely $L^\infty$
conductivities in \cite{AP}. However, the corresponding result for
the Schr\"odinger equation was not known until the recent
breakthrough \cite{Bu}.  As for the uniqueness in determining two
coefficients, see \cite{ChengYama}.
In \cite{KU} it is shown in two dimensions
that one can uniquely determine the magnetic field and the electrical potential from the DN map associated to the Pauli Hamiltonian.

If the DN map is measured only on a part of the boundary, then much 
less is known.  We only review here the results where no a-priori 
information is assumed.
In dimensions $n\ge 3$ a global result is shown in \cite{BuU}
where partial measurements of the DN map are assumed:  
More precisely, for $C^2$ conductivities if we measure the DN map
restricted to a slightly larger than the half of
the boundary, then one
can determine uniquely the potential. The proof relies on a Carleman
estimate with a linear weight function. The Carleman
estimate can also be used to construct complex geometrical optics
solutions for the Schr{\"o}dinger equation.
In \cite{K} the regularity assumption on the conductivity was relaxed to
$C^{3/2+\ell}$ with some $\ell > 0$. 
Stability estimates for the uniqueness result of \cite{BuU}
were given in \cite{HW}. Stability estimates for the magnetic
Schr\"odinger operator with partial data
in the setting of \cite{BuU} can be found in \cite{T}.

In \cite{KSU}, the result in \cite{BuU} was generalized and it is 
shown that
by all possible pairs of Dirichlet data on an arbitrary open subset
$\Gamma_+$ of the boundary and Neumann data on a slightly larger
open subset  than $\partial\Omega\setminus \Gamma_+$, one can
uniquely determine the potential.  The case of the magnetic
Schr\"odinger equation was considered in \cite{DKSjU} and
improvement on the regularity of the coefficients can be found in
\cite{KS}.

In this paper we show a result similar to \cite{KSU} in two dimensions
by constructing complex geometrical optics solutions with degenerate
weights.
We note that in two dimensions the problem is formally determined
while in three or higher dimensions it is overdetermined.
We now state the main result more precisely.

Let $\Omega\subset\R^2$ be a simply connected bounded domain with
smooth boundary.
The electrical conductivity of $\Omega$ is
represented by a bounded and positive function $\gamma(x)$.
In the absence
of sinks or sources of current, the potential $u\in
H^1(\Omega)$ with given boundary voltage potential $f\in
H^{\frac{1}{2}}(\partial\Omega)$ is a solution of the Dirichlet boundary value
problem
\begin{equation}\label{eq:0.2}
\begin{array}{rcl}
\mbox{div}(\gamma\nabla u) & = & 0 \mbox{ in }\Omega, \\
u\big|_{\partial\Omega} & = & f.
\end{array}
\end{equation}
The Dirichlet to Neumann (DN) map, or voltage to current map, is
given by
\begin{equation}\label{eq:0.3}
\Lambda_\gamma(f)= \gamma\frac{\partial
u}{\partial \nu }\Big|_{\partial\Omega},
\end{equation}
where $\nu$ denotes the unit outer normal to $\partial\Omega$.
This problem can be reduced to studying the set of Cauchy data for
the Schr\"odinger equation  with the potential $q$ given by:
\begin{equation}\label{eq:2.2}
q=\frac{\Delta\sqrt{\gamma}}{\sqrt{\gamma}}.
\end{equation}

\begin{equation}
\widetilde C_q=\left\{\left(u|_{\partial\Omega},
\frac{\partial u}{\partial \nu}\Big|_{\partial\Omega}\right)\mid
(\Delta+q) u= 0\hbox{ on }\Omega,\ u\in H^1(\Omega)\right\}.
\end{equation}
We have
$\widetilde C_q\subset H^{\frac{1}{2}}(\partial\Omega)\times
H^{-\frac{1}{2}}(\partial\Omega)$.

By using a conformal map, thanks to the Kellog-Warchawski theorem
(see e.g. p. 42 \cite{Po}), without loss of generality we assume that
$\Omega= \{ x\in \R^2\vert \thinspace \vert x\vert < 1\}$.

Let $\Gamma_-=\{(\cos\theta,\sin\theta)\vert \theta\in
(-\theta_0,\theta_0)\}$ be a connected subdomain in $\partial\Omega$
and $\theta_0\in (0,\pi),$ $\widehat x_\pm$ the boundary of
$\Gamma_-$: $\partial\Gamma_- =\{\widehat x_\pm\}.$ Denote
$\Gamma_+=S^1\setminus \Gamma_-.$ Let $\epsilon>0$ be a small number
such that $\theta_0+\epsilon\in (0,\pi].$ Denote by
$\Gamma_{-,\epsilon}=\{(\cos\theta,\sin\theta)\vert \theta\in
(-\theta_0-\epsilon,\theta_0+\epsilon)\}$ and by $\widehat
x_{\pm,\epsilon}$ the endpoints of $\Gamma_{-,\epsilon}.$

We have

\begin{theorem}\label{main}
Let $q_j \in C^{1+\ell}(\overline \Omega)$, $j=1,2$ for some positive $\ell.$ Consider the
following sets of partial Cauchy data:

\begin{equation}
\mathcal{C}_{q_j,\epsilon}=\left\{\left(u|_{\Gamma_{+}},
\frac{\partial u}{\partial \nu}\Big|_{\Gamma_{-,\epsilon}}\right)\mid
(\Delta+q_j) u= 0\hbox{ in }\Omega,\,\, u\vert_{\Gamma_-}=0,\,\,\ u\in H^1(\Omega)\right\},
\quad j=1,2.
\end{equation}
Assume
$$
\mathcal{C}_{q_1,\epsilon}= \mathcal{C}_{q_2,\epsilon}
$$
with some $\epsilon>0$.
Then
$$
q_1=q_2.
$$
\end{theorem}

As a direct consequence of Theorem~\ref{main} we have

\begin{corollary}\label{coro}
Let
$\gamma_1,\gamma_2$ be strictly positive functions and  there exists some positive number $\ell$ such that $\gamma_1,\gamma_2 \in C^{3+\ell}(\overline\Omega).$
Assume that $\gamma_1=\gamma_2$ on $\partial\Omega$ and
$$
{\gamma_1}\frac{\partial u}{\partial \nu}={\gamma_2}\frac{\partial u}
{\partial \nu}\quad\mbox{on}\quad
\Gamma_{-,\epsilon}
\mbox{ for all }
u\in H^{1\over 2}(\partial\Omega ),\,\,\supp \, u\subset\Gamma_+.
$$
Then $\gamma_1=\gamma_2$.
\end{corollary}

The proof of Theorem \ref{main} uses Carleman estimates for the
Laplacian with
degenerate limiting Carleman weights.  The results of \cite{BuU}
and \cite{KSU} use complex geometrical optics solutions of the form
\begin{equation}\label{cgo}
u= e^{\tau(\varphi+\sqrt{-1}\psi)} (a+r),
\end{equation}
where $\nabla \varphi \cdot\nabla \psi=0,
|\nabla \varphi|^2=|\nabla \psi|^2$
and $\varphi$ is a limiting Carleman weight and $a$ is smooth and
non-vanishing and
$\Vert r\Vert_{L^2(\Omega)} =O(\frac{1}{\tau})$, $\Vert r\Vert
_{H^1(\Omega)} =O(1)$. Examples of limiting Carleman weights are the linear
phase $\varphi(x)=x\cdot\omega$ with $\omega\in S^{n-1}$ which was
used in \cite{BuU}, and the non-linear phase
$\varphi(x)=
\ln|x-x_0|$, where  $x_0\in {\bf R}^n\setminus\overline{\Omega}$ which
was used in \cite{KSU}.
For a complete characterization of possible local Carleman weights in
the Euclidean space and more general manifolds, see \cite{DKSaU}.

In two dimensions
the limiting Carleman weights are harmonic functions so that there is
a larger class of complex geometrical optics solutions.
This freedom was used in
\cite{UW} to determine inclusions for a large class of systems in two
dimensions. In particular, one can use the harmonic function
$\varphi=\mbox{Re}\,z^n$ as limiting Carleman weight, assuming that $0$ 
is outside the domain.

In this paper we construct complex geometrical optics solutions  of the
form
\begin{equation}
u= e^{\tau(\varphi+\sqrt{-1}\psi)} (a+r)+ u_r,
\end{equation}
where $u_r$ is a ``reflected" term to guarantee that the solution
vanishes in particular subsets of the boundary, $\varphi$ is a
harmonic function having a finite number of non-degenerate critical
points in $\Omega$, and $\psi$ is the corresponding conjugate
harmonic function. However we need to modify the form with $\varphi$
harmonic but having non-degenerate critical points. Solutions as in
(\ref{cgo}) with degenerate harmonic functions were also used in
\cite{Bu} but here the phase function needs to satisfy further
restrictions in order to use them for the partial data problem.
Another complication is that the correction term $r$ and the
reflected term $u_r$ do not have the same asymptotic behavior in
$\tau$ as in \cite{KSU} because of the degeneration of the phase,
so that one needs to further decompose these terms and analyze
their asymptotic behavior in $\tau.$ See section 3 for more details.
In section 2 we prove a general Carleman estimate with degenerate
weights. Finally in section 4 we prove Theorem \ref{main}.

\section{\bf Carleman estimates with degenerate weights}

Throughout the paper we use the following notations:
\\

\noindent {\bf Notations}
$i=\sqrt{-1}$, $x_1, x_2, \xi_1, \xi_2 \in \R$,
$z=x_1+ix_2$, $\zeta=\xi_1+i\xi_2,$ $\frac{\partial}{\partial z}=\frac
12(\partial_{x_1}-i\partial_{x_2})$,
$\frac{\partial}{\partial {\overline z}}=\frac
12(\partial_{x_1}+i\partial_{x_2}),$ $H^{1,\tau}(\Omega)$ denotes
the space $H^1(\Omega)$ with norm $\Vert
v\Vert^2_{H^{1,\tau}(\Omega)}=\Vert
v\Vert^2_{H^{1}(\Omega)}+\tau^2\Vert v\Vert^2_{L^2(\Omega)}.$
The tangential derivative on the boundary is given by
$\partial_\tau=\nu_2\frac{\partial}{\partial x_1}
-\nu_1\frac{\partial}{\partial x_2},$
with $\nu=(\nu_1, \nu_2)$ the unit outer normal to $\partial\Omega,$
$B(\widehat x,\delta)=\{x\in \R^2\vert \vert x-\widehat x\vert <
\delta\},$ $S^1=\{x\in \R^2\vert \vert x\vert=1\}$,
 $f(x):\R^2\rightarrow \R^1$, $f''$ is the Hessian matrix  with entries
 $\frac{\partial^2 f}{\partial x_i\partial x_j}.$

Let $\Phi(z)=\varphi_1(x_1,x_2)+i\varphi_2(x_1,x_2)$  be a
holomorphic function in a domain $\Omega_0$, given that $\overline{\Omega}\subset\Omega_0$,
\begin{equation}\label{zzz}
\frac{\partial\Phi(z)}{ \partial \overline z}=0\quad \mbox{in}
\,\,\Omega_0,\quad \Phi\in C^2(\overline\Omega_0).
\end{equation}
Denote by $\mathcal H$ the set of critical points of a function $\Phi$
$$
\mathcal H=\left\{z\in\overline\Omega\vert \frac{\partial \Phi}
{\partial z} (z)=0 \right\}.
$$
Assume that $\Phi$ has no critical points at the boundary and
nondegenerate critical points in the interior;
\begin{equation}\label{mika}
\mathcal H\cap \partial\Omega=\{\emptyset\},\quad
\Phi^{''}(z)\ne 0 \quad \forall z\in \mathcal H.
\end{equation}
Then $\Phi$ we have only a finite number of critical points:
\begin{equation}\label{mona}
\mbox{card}\thinspace \mathcal H<\infty.
\end{equation}
Denote
$\frac{\partial \Phi}{\partial z}(z)
=\psi_1(x_1,x_2)+i\psi_2(x_1,x_2).$

We will prove Carleman estimates for the conjugated operator

\begin{equation}\nonumber
\Delta_{\tau}=e^{\tau \varphi_1} \Delta e^{-\tau \varphi_1}.
\end{equation}
We will use the factorization
\begin{equation}\label{factorization}
\Delta_\tau \widetilde v=
\left(2\frac{\partial}{\partial z}-\tau\frac{\partial\Phi}{\partial
z}\right)
\left(2\frac{\partial}{\partial\overline z}
-\tau\frac{\partial\overline\Phi}
{\partial \overline z}\right)\widetilde v
= \left(2\frac{\partial}{\partial\overline z}
-\tau\frac{\partial\overline\Phi} {\partial \overline
z}\right)
\left(2\frac{\partial}{\partial z}-\tau\frac{\partial\Phi}{\partial
z}\right)\widetilde v
\end{equation}
and prove Carleman estimates first for every term in the factorization.



\begin{proposition} \label{Proposition 2.1}
Let $\Phi$ satisfy (\ref{zzz}) and (\ref{mika}). Let
$\widetilde f\in L^2(\Omega)$, and $\widetilde v$ be solution
to the problem
\begin{equation}\label{zina}
2\frac{\partial \widetilde v}{\partial z} -\tau
\frac{\partial\Phi}{\partial z}\widetilde v=\widetilde f\quad
\mbox{ in }\,\Omega
\end{equation}
or $\widetilde v$ be solution to the problem
\begin{equation}\label{zina1}
2\frac{\partial \widetilde v}{\partial \overline z}-\tau
\frac{\partial\overline\Phi}{\partial\overline z}\widetilde v=\widetilde
f\quad\mbox{ in }\,\Omega.
\end{equation}
In the case  (\ref{zina}) we have
\begin{eqnarray}\label{vika1}
\left\Vert \left( \frac{\partial}{\partial
x_1}-i\psi_2\tau\right)\widetilde
v\right\Vert^2_{L^2(\Omega)}
- \tau\int_{\partial\Omega}(\nabla\varphi_1,\nu)\vert
\widetilde v\vert^2d\sigma\nonumber\\
+ \mbox{Re}\int_{\partial\Omega}i\left(\left(\nu_2
\frac{\partial}{\partial x_1}-\nu_1 \frac{\partial}{\partial x_2}\right)
\widetilde v\right)\overline{\widetilde
v}d\sigma
+ \left\Vert \left(i \frac{\partial}{\partial x_2}+\tau\psi_1\right)
\widetilde
v\right\Vert^2_{L^2(\Omega)}=\Vert\widetilde f\Vert^2_{L^2(\Omega)},
\end{eqnarray}
while in the case (\ref{zina1}) we have
\begin{eqnarray} \label{vika2}
\left\Vert \left(\frac{\partial}{\partial
x_1}+i\psi_2\tau\right)\widetilde
v\right\Vert^2_{L^2(\Omega)}
-\tau\int_{\partial\Omega}(\nabla\varphi_1,\nu)\vert
\widetilde v\vert^2d\sigma
+ \mbox{Re}\int_{\partial\Omega}i\left(\left(-\nu_2
\frac{\partial}{\partial x_1}+\nu_1 \frac{\partial}{\partial
x_2}\right)\widetilde v\right)\overline{\widetilde
v}d\sigma\nonumber\\
+ \left\Vert \left(i \frac{\partial}{\partial x_2}-\psi_1\tau\right)
\widetilde
v\right\Vert^2_{L^2(\Omega)}=\Vert\widetilde f\Vert^2_{L^2(\Omega)}.
\end{eqnarray}
\end{proposition}

\begin{proof}
We prove the statement of the proposition first for the
equation $ 2\frac{\partial \widetilde v}{\partial z}-\tau
\frac{\partial\Phi}{\partial z}\widetilde v=\widetilde f. $ Since
$2\frac{\partial }{\partial z}-\tau  \frac{\partial\Phi}{\partial
z}=( \frac{\partial}{\partial
x_1}-i\psi_2\tau)+(\frac{\partial}{i\partial x_2}-\psi_1\tau)$,
taking the $L^2-$ norms of the right and left hand sides of (\ref{zina})
we have

\begin{eqnarray}
\left\Vert \left(\frac{\partial}{\partial x_1}-i\psi_2\tau\right)
\widetilde
v\right\Vert^2_{L^2(\Omega)}
+ 2\mbox{Re}\left(\left( \frac{\partial}{\partial
x_1}-i\psi_2\tau\right)\widetilde v,\left(-i \frac{\partial}{\partial
x_2}-\psi_1\tau\right)\widetilde
v\right)_{L^2(\Omega)}\nonumber\\
+ \left\Vert \left(-i \frac{\partial}{\partial x_2}-\psi_1\tau\right)
\widetilde
v\right\Vert^2_{L^2(\Omega)}=\Vert\widetilde f\Vert^2_{L^2(\Omega)}.
\nonumber
\end{eqnarray}

Since we take the commutator to have $[(\frac{\partial}{\partial
x_1}-i\psi_2\tau),(\frac{\partial}{i\partial x_2}-\psi_1\tau)]\equiv
0$, we obtain
\begin{eqnarray}
\left\Vert \left( \frac{\partial}{\partial
x_1}-i\psi_2\tau\right)\widetilde v\right\Vert^2_{L^2(\Omega)}
+ \left(\left(
\frac{\partial}{\partial x_1}-i\psi_2\tau\right)\widetilde
v,\overline{({-i\nu_2\widetilde
v}})\right)_{L^2(\partial\Omega)}
+ \left(\overline{\nu_1\widetilde v},
{\left(-i\frac{\partial}{\partial
x_2}-\psi_1\tau\right)\widetilde v}\right)
_{L^2(\partial\Omega)}\nonumber\\
+ \left\Vert \left(i \frac{\partial}{\partial x_2}+\psi_1\tau\right)
\widetilde
v\right\Vert^2_{L^2(\Omega)}
=\Vert\widetilde f\Vert^2_{L^2(\Omega)}.\nonumber
\end{eqnarray}
This equality implies
\begin{eqnarray}
\left\Vert \left(\frac{\partial}{\partial
x_1}-i\psi_2\tau\right)\widetilde
v\right\Vert^2_{L^2(\Omega)}
-\tau\int_{\partial\Omega}(\psi_1\nu_1-\psi_2\nu_2)\vert
\widetilde v\vert^2d\sigma
+ \int_{\partial\Omega}i\left(\left(\nu_2
\frac{\partial}{\partial x_1}-\nu_1 \frac{\partial}{\partial
x_2}\right)\widetilde v\right)\overline{\widetilde
v}d\sigma\nonumber\\
+ \left\Vert \left(i \frac{\partial}{\partial x_2}+\psi_1\tau\right)
\widetilde
v\right\Vert^2_{L^2(\Omega)}
= \Vert\widetilde f\Vert^2_{L^2(\Omega)}.\nonumber
\end{eqnarray}

Finally by (2.1) we observe that $\psi_1=\frac
12(\frac{\partial\varphi_1}{\partial x_1}
+\frac{\partial\varphi_2}{\partial x_2})
=\frac{\partial\varphi_1}{\partial x_1}$ and $\psi_2=\frac
12(\frac{\partial\varphi_2}{\partial x_1}
-\frac{\partial\varphi_1}{\partial
x_2})=-\frac{\partial\varphi_1}{\partial x_2}.$ Therefore from
the above equality, (\ref{vika1}) follows immediately.

 Now we prove the statement of the theorem for the equation
(\ref{zina1}).
Since
$2\frac{\partial }{\partial \overline z}-\tau
\frac{\partial\overline\Phi}{\partial \overline z}
=(\frac{\partial}{\partial
x_1}+ i\psi_2\tau)+(-\frac{\partial}{i\partial x_2}-\psi_1\tau)$,
taking the $L^2-$ norms of the right and left hand sides
of (\ref{zina1}) we have

\begin{eqnarray}
\left\Vert \left(\frac{\partial}{\partial x_1}+i\psi_2\tau\right)
\widetilde
v\right\Vert^2_{L^2(\Omega)}
+ 2\mbox{Re}\left(\left( \frac{\partial}{\partial
x_1}+i\psi_2\tau\right)\widetilde v, \left(i \frac{\partial}{\partial
x_2}-\psi_1\tau\right)\widetilde
v\right)_{L^2(\Omega)}\nonumber\\
+ \left\Vert \left(i \frac{\partial}{\partial x_2}-\psi_1\tau\right)
\widetilde
v\right \Vert^2_{L^2(\Omega)}
= \Vert\widetilde f\Vert^2_{L^2(\Omega)}.\nonumber
\end{eqnarray}

Since $[(\frac{\partial}{\partial
x_1}+i\psi_2\tau),(\frac{\partial}{i\partial
x_2}+\psi_1\tau)]\equiv 0$, we obtain
\begin{eqnarray}
\left\Vert
\left(\frac{\partial}{\partial x_1}+i\psi_2\tau\right)\widetilde
v\right\Vert^2_{L^2(\Omega)}
+ \left(\left( \frac{\partial}{\partial
x_1}+i\psi_2\tau\right)\widetilde v,\overline{({i\nu_2\widetilde
v}})\right)_{L^2(\partial\Omega)}
+ \left(\overline{\nu_1\widetilde v},
{\left(i\frac{\partial}{\partial
x_2}-\psi_1\tau\right)\widetilde v}\right)
_{L^2(\partial\Omega)}\nonumber\\
+ \left\Vert \left(i \frac{\partial}{\partial x_2}-\psi_1\tau\right)
\widetilde
v\right\Vert^2_{L^2(\Omega)}
= \Vert\widetilde f\Vert^2_{L^2(\Omega)}.\nonumber
\end{eqnarray}
This equality implies
\begin{eqnarray}
\left\Vert \left(
\frac{\partial}{\partial x_1}+i\psi_2\tau\right)\widetilde
v\right\Vert^2_{L^2(\Omega)}
-\tau\int_{\partial\Omega}(\psi_1\nu_1-\psi_2\nu_2)\vert
\widetilde v\vert^2d\sigma
+ \int_{\partial\Omega}i\left(\left(-\nu_2
\frac{\partial}{\partial x_1}+\nu_1 \frac{\partial}{\partial
x_2}\right)\widetilde v\right)\overline{\widetilde
v}d\sigma\nonumber\\
+ \left\Vert \left(i \frac{\partial}{\partial x_2}-\psi_1\tau\right)
\widetilde
v\right\Vert^2_{L^2(\Omega)}
= \Vert\widetilde f\Vert^2_{L^2(\Omega)}.\nonumber
\end{eqnarray}


Finally we observe that $\psi_1=\frac
12(\frac{\partial\varphi_1}{\partial x_1}
+\frac{\partial\varphi_2}{\partial
x_2})=\frac{\partial\varphi_1}{\partial x_1}$ and $\psi_2=\frac
12(\frac{\partial\varphi_2}{\partial x_1}
-\frac{\partial\varphi_1}{\partial
x_2})=-\frac{\partial\varphi_1}{\partial x_2}.$
Thus estimate (\ref{vika2}) follows immediately from the above 
equality, finishing the proof of the proposition.
\end{proof}



Let $u$ solve the boundary value problem

\begin{equation}\label{suno3}\Delta
u=f\quad\text{in}\,\,\Omega,\quad u\vert_{\partial\Omega}=0.
\end{equation}

Denote
$$
\partial\Omega_{+}=\{(x_1,x_2)\in \partial\Omega\vert
(\nabla\varphi_1,\nu)> 0\}
$$
and
$$
\partial\Omega_{-}=\{(x_1,x_2)\in \partial\Omega\vert
(\nabla\varphi_1,\nu)<0\}.
$$

The main result of this section is the following Carleman estimate
with degenerate weights.

 \begin{theorem}\label{Theorem 2.1}
Suppose that $\Phi$ satisfies (\ref{zzz}) and (\ref{mika}).
Let $f\in L^2(\Omega)$, and let $u$ be a solution to (\ref{suno3}) with
$u\in H^1(\Omega).$  Then there exist  positive
constants $C>0$  and $\tau_0$ such that for all $\tau\ge \tau_0$:
\begin{eqnarray}\label{suno4}
\tau\Vert
ue^{\tau\varphi_1}\Vert^2_{L^2(\Omega)}
+ \Vert ue^{\tau\varphi_1}\Vert^2_{H^1(\Omega)}
+ \tau^2\left\Vert\left\vert\frac{\partial\Phi}{\partial z} \right\vert
ue^{\tau\varphi_1}\right\Vert^2_{L^2(\Omega)}
- \tau\int_{\partial\Omega_-}(\nu,\nabla\varphi_1)\left\vert
\frac{\partial u}{\partial\nu}\right\vert^2e^{2\tau\varphi_1}d\sigma
\nonumber\\
\le C\left(\Vert fe^{\tau\varphi_1}\Vert^2_{L^2(\Omega)}+\tau
\int_{\partial\Omega_+}(\nu,\nabla\varphi_1)\left\vert \frac{\partial
u}{\partial\nu}\right\vert^2e^{2\tau\varphi_1}d\sigma\right).
\end{eqnarray}
\end{theorem}

\begin{proof} As indicated earlier we can take $\Omega$ to be the unit ball.
Denote $\widetilde v=ue^{\tau\varphi_1}.$ Without the loss of generality we may assume that $u$ is a real valued function.
By (\ref{factorization})
$$
\Delta_\tau \widetilde v=
\left(2\frac{\partial}{\partial z}-\tau\frac{\partial\Phi}{\partial
z}\right)\left(2\frac{\partial}{\partial\overline z}
-\tau\frac{\partial\overline\Phi}
{\partial \overline z}\right)\widetilde v
= \left(2\frac{\partial}{\partial\overline z}
-\tau\frac{\partial\overline\Phi} {\partial \overline
z}\right)\left(2\frac{\partial}{\partial z}-\tau\frac{\partial\Phi}
{\partial z}\right)\widetilde v = fe^{\tau\varphi_1}.
$$

Denote $\widetilde w_1=(2\frac{\partial}{\partial\overline z}
-\tau\frac{\partial\overline\Phi} {\partial \overline z})\widetilde v,
\widetilde w_2=(2\frac{\partial}{\partial z} -\tau\frac{\partial\Phi}
{\partial  z})\widetilde v$ and $\frac{\partial\Phi}{\partial
z}=\psi_1(x_1,x_2)+i\psi_2(x_1,x_2).$
Thanks to the boundary condition (\ref{suno3}), we have
$$
\widetilde w_1\vert_{\partial\Omega}=2\partial_{\overline z}\widetilde
v\vert_{\partial\Omega}=(\nu_1+i\nu_2)\frac{\partial
\widetilde v}{\partial \nu}\vert_{\partial\Omega},\,\, \widetilde
w_2\vert_{\partial\Omega}=2\partial_{ z}\widetilde
v\vert_{\partial\Omega}=(\nu_1-i\nu_2)\frac{\partial \widetilde
v}{\partial \nu}\vert_{\partial\Omega}.
$$

By Proposition \ref{Proposition 2.1}
\begin{eqnarray}
\left\Vert \left( \frac{\partial}{\partial
x_1}-i\psi_2\tau\right)\widetilde w_1\right\Vert^2_{L^2(\Omega)}
- \tau\int_{\partial\Omega}(\nabla\varphi_1,\nu)\left\vert
\frac{\partial\widetilde
v}{\partial\nu}\right\vert^2d\sigma
+ \mbox{Re}\int_{\partial\Omega}i\left(\left(\nu_2
\frac{\partial}{\partial x_1}-\nu_1 \frac{\partial}{\partial
x_2}\right)\widetilde w_1\right)\overline{\widetilde
w_1}d\sigma                      \nonumber\\
+ \left\Vert \left(i\frac{\partial}{\partial x_2}+\psi_1\tau\right)
\widetilde
w_1\right\Vert^2_{L^2(\Omega)}=\Vert
fe^{\tau\varphi_1}\Vert^2_{L^2(\Omega)}\nonumber
\end{eqnarray}
and
\begin{eqnarray}
\left\Vert \left(\frac{\partial}{\partial
x_1}+i\psi_2\tau\right)\widetilde w_2\right\Vert^2_{L^2(\Omega)}
- \tau\int_{\partial\Omega}(\nabla\varphi_1,\nu)\left\vert
\frac{\partial\widetilde
v}{\partial\nu}\right\vert^2d\sigma
+ \mbox{Re}\int_{\partial\Omega}i\left(\left(-\nu_2
\frac{\partial}{\partial x_1}+\nu_1 \frac{\partial}{\partial
x_2}\right)\widetilde w_2\right)\overline{\widetilde
w_2}d\sigma\nonumber\\
+ \left\Vert \left(i\frac{\partial}{\partial x_2}-\psi_1\tau\right)
\widetilde
w_2\right\Vert^2_{L^2(\Omega)}
= \Vert
fe^{\tau\varphi_1}\Vert^2_{L^2(\Omega)}.\nonumber
\end{eqnarray}

Let us simplify the integral $\mbox{Re}\thinspace
i\int_{\partial\Omega} \left(\left(\nu_2
\frac{\partial}{\partial x_1}-\nu_1 \frac{\partial}{\partial
x_2}\right)\widetilde w_1\right)\overline{\widetilde w_1}d\sigma.$
We recall that $\widetilde v=ue^{\tau \varphi_1}$ and $\widetilde
w_1=(\nu_1+i\nu_2)\frac{\partial \widetilde v}{\partial
\nu}=(\nu_1+i\nu_2)\frac{\partial u}{\partial \nu}
e^{\tau \varphi_1}.$  Thus
\begin{eqnarray}
\mbox{Re}\int_{\partial\Omega}i\left(\left(\nu_2
\frac{\partial}{\partial x_1}-\nu_1 \frac{\partial}{\partial
x_2}\right)\widetilde w_1\right)\overline{\widetilde
w_1}d\sigma=\nonumber\\
\mbox{Re}\int_{\partial\Omega}i\left(\left(\nu_2
\frac{\partial}{\partial
x_1}-\nu_1 \frac{\partial}{\partial x_2}\right)
\left[(\nu_1+i\nu_2)\frac{\partial
u}{\partial\nu}e^{\tau \varphi_1}\right]\right)
(\nu_1-i\nu_2)\frac{\partial u}{\partial\nu}e^{\tau\varphi_1}d\sigma=\nonumber\\
\mbox{Re}\int_{\partial\Omega}i\left[\left(\nu_2
\frac{\partial}{\partial
x_1}-\nu_1 \frac{\partial}{\partial x_2}\right)(\nu_1+i\nu_2)\right]
\left\vert\frac{\partial
\widetilde v}{\partial\nu}\right\vert^2
(\nu_1-i\nu_2)d\sigma+\nonumber\\
\mbox{Re}\int_{\partial\Omega}\frac 12i\left(\nu_2
\frac{\partial}{\partial x_1}-\nu_1 \frac{\partial}{\partial
x_2}\right)\left\vert\frac{\partial \widetilde v}{\partial\nu}
\right\vert^2 d\sigma=
\int_{\partial\Omega} \left\vert\frac{\partial \widetilde v}
{\partial\nu}\right\vert^2
d\sigma.\nonumber\end{eqnarray}

Let us simplify the integral
$\mbox{Re}\int_{\partial\Omega}i\left(\left(-\nu_2 \frac{\partial}
{\partial x_1}+\nu_1 \frac{\partial}{\partial x_2}\right)\widetilde w_2
\right)\overline{\widetilde
w_2}d\sigma.$ We recall that $\widetilde v=ue^{\tau \varphi_1}$ and
$\widetilde w_2=(\nu_1-i\nu_2)\frac{\partial \widetilde v}{\partial
\nu}=(\nu_1-i\nu_2)\frac{\partial u}{\partial \nu} e^{\tau
\varphi_1}.$ We conclude
\begin{eqnarray}
\mbox{Re}\int_{\partial\Omega}i\left(\left(-\nu_2 \frac{\partial}
{\partial
x_1}+\nu_1 \frac{\partial}{\partial x_2}\right)\widetilde w_2\right)
\overline{\widetilde w_2}d\sigma=\nonumber\\
\mbox{Re}\int_{\partial\Omega}i\left(\left(-\nu_2 \frac{\partial}
{\partial x_1}+\nu_1 \frac{\partial}{\partial x_2}\right)\left[
(\nu_1-i\nu_2)\frac{\partial u}{\partial\nu}e^{\tau \varphi_1}\right]\right)
(\nu_1+i\nu_2)\frac{\partial u}{\partial\nu}e^{\tau\varphi_1}d\sigma=\\
\mbox{Re}\int_{\partial\Omega}i\left[\left(-\nu_2 \frac{\partial}
{\partial x_1}+\nu_1 \frac{\partial}{\partial x_2}\right)(\nu_1-i\nu_2)\right]
\left\vert\frac{\partial
\widetilde v}{\partial\nu}\right\vert^2 (\nu_1+i\nu_2)d\sigma-\nonumber\\
\mbox{Re}\int_{\partial\Omega}\frac 12i\left(\nu_2
\frac{\partial}{\partial x_1}-\nu_1 \frac{\partial}{\partial
x_2}\right)\left\vert\frac{\partial \widetilde v}{\partial\nu}
\right\vert^2 d\sigma=
\int_{\partial\Omega} \left\vert\frac{\partial \widetilde v}
{\partial\nu}\right\vert^2
d\sigma.\nonumber
\end{eqnarray}
Using the above formulae we obtain
\begin{eqnarray} \label{suno5}
\left\Vert \left( \frac{\partial}{\partial
x_1}+i\psi_2\tau\right)\widetilde w_2\right\Vert^2_{L^2(\Omega)}
+ \left\Vert
\left(i\frac{\partial}{\partial x_2}-\psi_1\tau\right)\widetilde
w_2\right\Vert^2_{L^2(\Omega)}
- 2\tau \int_{\partial\Omega}(\nu,\nabla\varphi_1)\left\vert
\frac{\partial\widetilde v}{\partial\nu}\right\vert^2d\sigma\nonumber\\
+ \left\Vert \left( \frac{\partial}{\partial x_1}-i\psi_2\tau\right)
\widetilde
w_1\right\Vert^2_{L^2(\Omega)}
+ \left\Vert \left(i \frac{\partial}{\partial
x_2}+\psi_1\tau\right)\widetilde w_1\right\Vert^2_{L^2(\Omega)}
\nonumber\\
+ 2\int_{\partial\Omega} \left\vert\frac{\partial \widetilde v}
{\partial\nu}\right\vert^2 d\sigma = 2\Vert
fe^{\tau\varphi_1}\Vert^2_{L^2(\Omega)}.
\end{eqnarray}

Let a function $\widetilde\psi_k$ satisfy
$$
\frac{\partial\widetilde \psi_1}{\partial
x_1}=\psi_2,\quad\frac{\partial\widetilde \psi_2}{\partial
x_2}=\psi_1\quad \mbox{in}\,\,\Omega.
$$
We can rewrite equality (\ref{suno5}) in the form
\begin{eqnarray}\label{(2.20)}
\left\Vert  \frac{\partial}{\partial
x_1}(e^{i\widetilde\psi_1\tau}\widetilde w_2)\right\Vert^2_{L^2(\Omega)}
+ \left\Vert
\frac{\partial}{\partial x_2}(e^{i\widetilde\psi_2\tau}\widetilde
w_2)\right\Vert^2_{L^2(\Omega)}
- 2\tau\int_{\partial\Omega}(\nu,\nabla\varphi_1)\left\vert
\frac{\partial\widetilde v}{\partial\nu}\right\vert^2d\sigma\nonumber\\
+ \left\Vert  \frac{\partial}{\partial x_1}(e^{-i\widetilde\psi_1\tau}
\widetilde w_1)\right\Vert^2_{L^2(\Omega)}
+ \left\Vert \frac{\partial}{\partial
x_2}(e^{-i\widetilde\psi_2\tau}\widetilde w_1)\right\Vert^2_{L^2(\Omega)}
\nonumber\\
+ 2\int_{\partial\Omega} \left\vert\frac{\partial \widetilde v}
{\partial\nu}\right\vert^2 d\sigma = {2}\Vert
fe^{\tau\varphi_1}\Vert^2_{L^2(\Omega)}.
\end{eqnarray}
Observe that there exists some positive constant $C>0$, independent of
$\tau$ such that
\begin{eqnarray}\label{(2.21)}
\frac 1C(\Vert\widetilde w_1\Vert^2_{L^2(\Omega)}+\Vert\widetilde
w_2\Vert^2_{L^2(\Omega)})
\le \frac 12 \left\Vert  \frac{\partial}{\partial
x_1}(e^{i\widetilde\psi_2\tau}\widetilde w_2)\right\Vert^2_{L^2(\Omega)}
+ \frac12\left\Vert \frac{\partial}{\partial x_2}(e^{i\widetilde\psi_1\tau}
\widetilde w_2)\right\Vert^2_{L^2(\Omega)}\nonumber\\
- \tau
\int_{\partial\Omega_-}(\nu,\nabla\varphi_1)\left\vert
\frac{\partial\widetilde v}{\partial\nu}\right\vert^2d\sigma\nonumber\\
+ \frac 12\left\Vert \frac{\partial}{\partial
x_1}(e^{-i\widetilde\psi_1\tau}\widetilde w_1)\right\Vert^2
_{L^2(\Omega)}
+ \frac12\left\Vert \frac{\partial}{\partial x_2}
(e^{-i\widetilde\psi_2\tau}\widetilde w_1)\right\Vert^2_{L^2(\Omega)}.
\end{eqnarray}

Since $\widetilde v$ is a real-valued function, we have
$$
\left\Vert \frac{\partial \widetilde v}{\partial x_1}
+ \tau\psi_1\widetilde v\right\Vert^2_{L^2(\Omega)}
+ \left\Vert \frac{\partial \widetilde v}{\partial
x_2}-\tau\psi_2\widetilde v\right\Vert^2_{L^2(\Omega)}
\le C_0(\Vert\widetilde
w_1\Vert^2_{L^2(\Omega)}+\Vert\widetilde w_2\Vert^2_{L^2(\Omega)}).
$$
Therefore
\begin{eqnarray}\label{(2.22)}
\left\Vert \frac{\partial \widetilde v}{\partial
x_1}\right\Vert^2_{L^2(\Omega)}
- \tau\int_\Omega\left(\frac{\partial\psi_1}{\partial x_1}
- \frac{\partial\psi_2}{\partial x_2}\right)\widetilde v^2dx\nonumber
+\Vert\tau\psi_1\widetilde v\Vert^2_{L^2(\Omega)}\\
+ \left\Vert \frac{\partial
\widetilde v}{\partial x_2}\right\Vert^2_{L^2(\Omega)}
+ \Vert\tau\psi_2\widetilde
v\Vert^2_{L^2(\Omega)} \le C_1(\Vert\widetilde
w_1\Vert^2_{L^2(\Omega)}+\Vert\widetilde w_2\Vert^2_{L^2(\Omega)}).
\end{eqnarray}

By the Cauchy-Riemann equations, the second term of the left hand side 
of (\ref{(2.22)}) is zero.

Now since by assumption (\ref{mika}) the function $\Phi$ has only non degenerate critical points, we have
\begin{equation}\label{(2.23)}
\tau\Vert \widetilde v\Vert^2_{L^2(\Omega)}
\le C\left(\Vert \widetilde
v\Vert^2_{H^1(\Omega)}
+ \tau^2\left\Vert\left\vert\frac{\partial\Phi}{\partial
z} \right\vert \widetilde v\right\Vert^2_{L^2(\Omega)}\right).
\end{equation}

By (\ref{(2.22)}) and (\ref{(2.23)})
\begin{equation}\label{(2.24)}
\tau\Vert \widetilde v\Vert^2_{L^2(\Omega)}+\Vert \widetilde
v\Vert^2_{H^1(\Omega)}
+ \tau^2\left\Vert\left\vert\frac{\partial\Phi}{\partial
z} \right\vert \widetilde v\right\Vert^2_{L^2(\Omega)}\le C_1(\Vert\widetilde
w_1\Vert^2_{L^2(\Omega)}+\Vert\widetilde w_2\Vert^2_{L^2(\Omega)}).
\end{equation}
Using (\ref{(2.24)}), we obtain from (\ref{(2.20)}) and (\ref{(2.21)})
that 
\begin{eqnarray}
\frac{1}{C_5}\left(
\tau\Vert \widetilde v\Vert^2_{L^2(\Omega)}+\Vert \widetilde
v\Vert^2_{H^1(\Omega)}
+ \tau^2\left\Vert\left\vert\frac{\partial\Phi}{\partial
z} \right\vert\widetilde v\right\Vert^2_{L^2(\Omega)}\right)
- 2\tau\int_{\partial\Omega}(\nu,\nabla\varphi_1)\left\vert
\frac{\partial\widetilde v}{\partial\nu}\right\vert^2d\sigma\nonumber
\\
+2\int_{\partial\Omega}
\left\vert\frac{\partial \widetilde v}{\partial\nu}\right\vert^2
d\sigma
\le 2\Vert fe^{s\varphi_1}\Vert^2_{L^2(\Omega)}-\tau
\int_{\partial\Omega_-}(\nu,\nabla\varphi_1)\left\vert
\frac{\partial\widetilde v}{\partial\nu}\right\vert^2d\sigma\nonumber
\end{eqnarray}
concluding the proof of the theorem.
\end{proof}
We note that in the theorem we can add a zeroth order term to the
Laplacian and the estimate is valid for large enough $\tau.$

As usual the Carleman estimate implies the existence of solutions  
for the Schr\"odinger equation satisfying estimates with
appropriate weights.

Consider the following problem
\begin{equation}\label{(2.26)}
\Delta u+q_0u=f\quad \mbox{in}\,\,\Omega,\quad
u\vert_{\widetilde\Gamma}=g,
\end{equation}
where $\overline
{\widetilde\Gamma}\subset
\{x\in \partial\Omega\vert (\nu,\nabla\varphi_1)<0\}.$
We have
\begin{proposition}\label{Proposition 2.3}
Let $q_0\in L^\infty(\Omega).$ There exists
$\tau_0>0$ such that for all $\tau>\tau_0$  there exists a solution
to problem (\ref{(2.26)}) such that \begin{equation} \label{(2.27)}\Vert
ue^{-\tau\varphi_1}\Vert_{L^2(\Omega)}\le C(\Vert
fe^{-\tau\varphi_1}\Vert_{L^2(\Omega)}+\Vert ge^{-\tau\varphi_1}
\Vert_{L^2(\widetilde \Gamma)})/\tau^\frac 12,
\end{equation}
\end{proposition}

\begin{proof}
Let us introduce the space
$$
H= \left\{v\in H^1_0(\Omega)\vert \thinspace \Delta v+q_0v\in
L^2(\Omega),\frac{\partial v}{\partial
\nu}\vert_{\partial\Omega_+}=0 \right\}
$$
with the scalar product
$$
(v_1,v_2)_H=\int_\Omega e^{2\tau\varphi_1}(\Delta
v_1+q_0v_1)(\Delta v_2+q_0v_2)dx.
$$
By  Proposition \ref{Theorem 2.1} $H$ is
a Hilbert space. Consider the linear functional
on $H:$ $v\rightarrow \int_\Omega vfdx+\int_{\tilde\Gamma}g\frac{\partial v}{\partial\nu}d\sigma.$ By (\ref{suno4}) this is the
continuous linear functional with the norm estimated by a constant
$C(\Vert fe^{\tau\varphi_1}\Vert_{L^2(\Omega)}+\Vert ge^{\tau\varphi_1}\Vert_{L^2(\tilde \Gamma)})/\tau^{\frac{1}{2}}.$
Therefore by the Riesz theorem there exists an
element $\widehat v \in H$ so that
$$
\int_\Omega vfdx+\int_{\tilde\Gamma}g\frac{\partial v}{\partial\nu}d\sigma=\int_\Omega e^{2\tau\varphi_1}(\Delta\widehat v
+q_0\widehat v)(\Delta v+q_0v)dx.
$$
Then, as a solution to (\ref{(2.26)}), we take the function
$u=e^{2\tau\varphi_1}(\Delta\widehat v+q_0\widehat v).$
\end{proof}

\section{\bf Complex geometrical optics solutions with degenerate
weights}

In this section we construct the complex geometrical optics solutions  which
will play the critical role in the proof of Theorem 1.1.

We first observe that we can put the sets $\Gamma_{-}$ and 
$\partial\Omega \setminus \Gamma_{-,\epsilon}$ in a more convenient 
position on the boundary of the unit ball
and  slightly deform the ball itself.

Namely we set

\begin{equation}\label{domain1}
\Omega\subset  B(0,1), \quad \Gamma_-\subset S^1, \quad\mathcal
S \equiv\partial\Omega\setminus\Gamma_{-,\epsilon}\subset S^1.
\end{equation}
Let $\ell_+\in \Gamma_+$  be a piece of $\partial\Omega$ between the
points $\hat x_+$ and $\hat x_{+,\epsilon}$ and  $\ell_-\in
\Gamma_+$ be a piece of $\partial\Omega$ between the points $\hat x_-$
and $\hat x_{-,\epsilon}.$ Then
\begin{equation}\label{domain2}
\ell_\pm
 \subset B(0,1).
\end{equation}
We construct complex geometrical optics solutions of the Schr\"odinger 
equation
$\Delta+q_1$, with $q_1$ satisfying the conditions of Theorem
\ref{main}. Consider the equation
\begin{equation}\label{2.1I}
L_1u=\Delta u+q_1u=0 \quad \text{in}\,\, \Omega.
\end{equation}
Let $\Phi(z)$ be a holomorphic function satisfying (\ref{zzz}) and
(\ref{mika}). Let us fix small positive constants $\epsilon,$
$\epsilon'$ and consider two domains:
\begin{equation}\label{(2.28)}
\partial\Omega_{-,-\epsilon}=\{x\in\partial\Omega\vert
(\nabla\varphi_1,\nu)<-\epsilon\},\quad
\partial\Omega_{+,\epsilon'}=\{x\in\partial\Omega\vert
(\nabla\varphi_1,\nu)>\epsilon'\}.
\end{equation}
Suppose that
\begin{equation}\label{mmm}
\overline \Gamma_-\subset \partial\Omega_{-,-\epsilon},
\end{equation}
and

\begin{equation}\label{mmm1}
\overline {\mathcal S}\subset \partial\Omega_{+,\epsilon'}.
\end{equation}

We will construct solutions to (\ref{2.1I}) of the form
\begin{equation}\label{mozila}
u_1(x)=e^{\tau\Phi(z)}a(z)-\chi_1(x)
e^{\tau\Phi(\frac{1}{\overline z})}
a\left(\frac{1}{\overline z}\right)
+ e^{\tau\Phi}u_{11}+e^{\tau\varphi_1}u_{12},
\quad u_1\vert_{\Gamma_-}=0.
\end{equation}

We explain in the next subsections the different phase function
$\varphi_1$ and the amplitude $a(z)$ in (3.7).
 Moreover we derive the behavior for large $\tau$ of the different
pieces of the complex geometrical optics solutions.

\subsection{The amplitude $a(z)$ and the function $\chi_1$}

The amplitude
$a(z)$  has the following properties:
$$
a\in C^2(\overline\Omega),\quad \frac{\partial a}{\partial \overline z}
\equiv 0, \quad a(z)\ne 0\,\,\mbox{on}\,\,\overline \Omega.
$$
Next we construct the cut-off function $\chi_1(x)$.

By (\ref{domain1}) and (\ref{domain2}), there exists a neighborhood
$\mathcal O_1$ of the set $\Gamma_-$ such that $\widetilde
\varphi_1(x)=\mbox{ Re}\, \Phi(\frac{1}{\overline z})$ is a harmonic
function satisfying
\begin{equation}\label{(2.29)}
\widetilde \varphi_1(x)<\varphi(x),\quad \,\,
\forall x\in \Omega\cap \mathcal O_1,
\end{equation}
\begin{equation}\label{(2.30)} \partial \Omega\cap \mathcal O_1\subset
\partial\Omega_{-,-\frac{\epsilon}{2}},
\end{equation}
\begin{equation}\label{kisa11}
\mbox{supp}\,\nabla\chi_1\subset\subset B(0,1)\cap {\mathcal O}_1.
\end{equation}
Consider the following integral
$$
J(\tau)=\int_{\Omega}  \chi_1 r(x) e^{\tau
\Phi(\frac{1}{\overline z})-\tau\overline{\Phi(z)}}dx.
$$
We have
\begin{proposition}\label{miura} Let $r\in C^{1+\ell
}(\overline \Omega)$ for some positive $\ell.$ Then
$$
J(\tau)=o\left(\frac 1\tau\right).
$$
\end{proposition}

\begin {proof} Observe that the function $\chi_1$ can be chosen in
such a way that
\begin{equation}\label{supp}
\partial_{\overline z}\left(\Phi\left(\frac {1}{\overline z}\right)
- \overline{\Phi(z)}\right)\vert_{\mbox{supp} \chi_1}\ne 0.
\end{equation}

Assume that for some point from $\partial\Omega_{-,-\epsilon}$ we have
$$
\partial_{\overline z}\left(\Phi\left(\frac {1}{\overline z}\right)
-\overline{\Phi( z)}\right)\vert_{\mbox{supp} \chi_1}=0,
$$
and the above equality is equivalent to
$$
\mbox{Re} (\Phi'(z)z)=0.
$$ 
This equality and the Cauchy-Riemann equations imply that 
$\frac{\partial \varphi}{\partial\nu}=0$ at this point, which is a 
contradiction.
Since it suffices to choose $\supp \chi_1$ close to $\Gamma_-$,
the proof of (\ref{supp}) is completed.

Therefore
$$
J(\tau)=\int_{\Omega}  \chi_1 r(x) e^{\tau
\Phi(\frac{1}{\overline z})-\tau\overline{\Phi(z)}}dx
= \frac 1\tau\int_{\Omega}  \chi_1 r(x) \frac{1}{\partial_{\overline z}
(\Phi(\frac {1}{\overline z})-\overline{\Phi(z)})}
\partial_{\overline z}e^{\tau
\Phi(\frac{1}{\overline z})-\tau\overline{\Phi(z)}}dx.
$$
Integrating by parts we have:
$$
J(\tau)=-\frac 1\tau\int_{\Omega}  \partial_{\overline z}(\chi_1 r(x))
\frac{1}{\partial_{\overline z} (\Phi(\frac {1}{\overline z})
-\overline{\Phi(z)})}
e^{\tau
\Phi(\frac{1}{\overline z})-\tau\overline{\Phi(z)}}dx
$$
$$+\frac {1}{2\tau}\int_{\partial\Omega}  \chi_1 r(x)
\frac{1}{\partial_{\overline z}(\Phi(\frac {1}{\overline z})
-\overline{\Phi(z)})}(\nu_1+i\nu_2)e^{\tau
\Phi(\frac{1}{\overline z})-\tau\overline{\Phi(z)}}d\sigma=J_1+J_2.
$$
Observe that on $\partial\Omega$
$$
e^{\tau \Phi(\frac{1}{\overline z})-\tau\overline{\Phi(z)}}
= e^{2\tau i\mbox{Im}\, \Phi(z)}.
$$
Using a stationary phase and taking into account that $\partial_\nu 
\mbox{Re}\Phi
= \partial_\tau\mbox{Im}\Phi\ne 0$
on $\mbox{supp} \chi_1\cap \partial\Omega$, we obtain
$$
J_2=o\left(\frac 1\tau\right).
$$
Next we observe that since $r\in C^{1+\ell}(\overline \Omega)$
we have
$$
J_1 =o\left(\frac 1\tau\right).
$$
The proof of the proposition is finished.
\end{proof}

\subsection{Construction of $u_{11}$}

The function $e^{\tau\Phi(z)}a(z)-\chi_1(x)e^{\tau\Phi(\frac{1}
{\overline z})}a(\frac{1}{\overline z})$ does not satisfy equation 
(\ref{2.1I}).
We construct $u_{11}$ in the next term in the asymptotic expansion. 
Before we start the construction of this term we need
several propositions.

Let us introduce the operators:

\begin{eqnarray}
\partial_{\overline z}^{-1}g=\frac{1}{2\pi i}\int_\Omega
\frac{g(\zeta, \overline\zeta)}{\zeta-z}
d\zeta\wedge d\overline\zeta=-\frac 1\pi\int_\Omega
\frac{g(\zeta,\overline\zeta)}{\zeta-z}d\xi_1d\xi_2,\nonumber\\\quad \partial_{
z}^{-1}g=-\frac{1}{2\pi i}\overline{\int_\Omega \frac{\overline
g(\zeta,\overline\zeta)}{\zeta-z} d\zeta\wedge d\overline\zeta}=-\frac 1\pi\int_\Omega
\frac{g(\zeta,\overline\zeta)}{\overline\zeta-\overline z}d\xi_1d\xi_2.
\end{eqnarray}

Then we know (e.g., \cite{VE} p. 56):
\begin{proposition}\label{Proposition 3.0}
Let $m\ge 0$ be an integer number and $\alpha\in (0,1).$ The operators
$\partial_{\overline z}^{-1},\partial_{ z}^{-1}\in \mathcal
L(C^{m+\alpha}(\overline \Omega),C^{m+\alpha+1}(\overline \Omega)).$
\end{proposition}

Here and henceforth $\mathcal L(X,Y)$ denotes the Banach
space of all bounded linear operators from a Banach space $X$ to
another Banach space $Y$.

We define two other operators:
\begin{equation}\label{(3.1)}
R_{\Phi}g=e^{\tau(\overline {\Phi(z)}-{\Phi(z)})}\partial_{\overline
z}^{-1}(ge^{\tau( {\Phi(z)}-\overline {\Phi(z)})}),\quad \widetilde
R_{\Phi}g=e^{\tau(\overline {\Phi(z)}-{\Phi(z)})}\partial_{
z}^{-1}(ge^{\tau( {\Phi(z)}-\overline {\Phi(z)})}).
\end{equation}

\begin{proposition}\label{Proposition 3.1} Let $g\in
C^\epsilon(\overline\Omega)$ for some positive $\epsilon.$ The function
$R_{\Phi}g$ is a solution to
\begin{equation}\label{(3.2)}
\partial_{\overline z}R_{\Phi}g-\tau\overline{\frac{\partial\Phi(z)}
{\partial z}}R_{\Phi}g=g\quad\mbox{in}\,\,\Omega.
\end{equation}
The function $\widetilde R_{\Phi}g$ solves
\begin{equation}\label{(3.3)}
\partial_{ z}\widetilde R_{\Phi}g+\tau \frac{\partial\Phi(z)}
{\partial z}\widetilde R_{\Phi}g=g\quad\mbox{in}\,\,\Omega.
\end{equation}
\end{proposition}
\begin{proof}
The proof is by direct computations:
\begin{eqnarray}
\partial_{ z}\widetilde R_{\Phi}g+\tau \frac{\partial\Phi(z)}{\partial z}
\widetilde R_{\Phi}g=
\partial_z(e^{\tau(\overline {\Phi(z)}-{\Phi(z)})}
\partial_{ z}^{-1}(ge^{\tau( {\Phi(z)}-\overline
{\Phi(z)})}))\nonumber\\+\tau \frac{\partial\Phi(z)}{\partial z}
(e^{\tau(\overline {\Phi(z)}-{\Phi(z)})}
\partial_{\overline z}^{-1}(ge^{\tau( {\Phi(z)}-\overline
{\Phi(z)})}))=\nonumber\\-\tau\frac{\partial\Phi(z)}{\partial z}
(e^{\tau(\overline {\Phi(z)}-{\Phi(z)})}
\partial_{ z}^{-1}(ge^{\tau( {\Phi(z)}-\overline
{\Phi(z)})}))+(e^{\tau(\overline {\Phi(z)}-{\Phi(z)})} (ge^{\tau(
{\Phi(z)}-\overline {\Phi(z)})}))\nonumber\\+\tau
\frac{\partial\Phi(z)}{\partial z} (e^{\tau(\overline
{\Phi(z)}-{\Phi(z)})}
\partial_{\overline z}^{-1}(ge^{\tau( {\Phi(z)}-\overline
{\Phi(z)})}))=g.\nonumber
\end{eqnarray}
\end{proof}

Denote
$$
\mathcal O_\epsilon=\{ x\in \Omega\vert
dist(x,\partial\Omega)\le \epsilon\}.
$$

\begin{proposition}\label{Proposition 3.2}
Let $g\in C^1(\Omega), g\vert_{\mathcal O_\epsilon}\equiv 0, g(x)\ne 0 $
for
all $x\in\mathcal H.$ Then
\begin{equation}\label{(3.4)}\vert R_{\Phi}g(x)\vert
+ \vert\widetilde R_{\Phi} g(x)\vert\le C \max_{x\in\mathcal H}\vert g(x)
\vert/\tau
\end{equation} for all $x\in \mathcal O_{\epsilon/2}.$ If $g\in
C^2(\overline \Omega)$ and $g\vert_{\mathcal H}=0$ then
\begin{equation}\label{(3.5)}\vert R_{\Phi} g(x)\vert+\vert
\widetilde R _{\Phi}g(x)\vert\le C /\tau^2\end{equation}
for all $x\in \mathcal O_{\epsilon/2}.$
\end{proposition}

\begin{proof}
Observe that $e^{\tau(
\Phi(z)-\overline{\Phi(z)})}=e^{2i\tau\mbox{Im}\Phi(z) }.$ By the
Cauchy-Riemann equations, the sets of the critical points of
$\Phi(z)$ and $\mbox{Im}\Phi(z)$ are exactly the same.
Therefore by our
assumptions the Hessian of $\mbox{Im}\Phi(z)$ is nondegenerate at each
point of $\mathcal H$ and it is enough to show that
$$
\left\vert\int_\Omega e^{2i\tau\mbox{Im}\Phi(z)} 
\frac{g(\zeta,\overline\zeta)}{z-\zeta} d\zeta\wedge
d\overline\zeta\right\vert
\le C\max_{x\in\mathcal H}\vert g(x)\vert/\tau\quad
\mbox{and}\,\,\left\vert\int_\Omega e^{2i\tau\mbox{Im}\Phi(z)}  
\frac{g(\zeta, \overline\zeta)}{z-\zeta}
d\zeta\wedge
d\overline\zeta\right\vert\le C/\tau^2.
$$
We observe that for any $z=x_1+ix_2\in \mathcal O_{\frac \epsilon
2}$ the function $ \frac{g(\zeta)}{z-\zeta}$ in the variable 
$\zeta$ is smooth and compactly supported.
The statement of the proposition follows from the standard stationary 
phase argument (see e.g., \cite{H}).
\end{proof}

Denote
\begin{equation}\label{polynomial}
r(z)=\Pi_{k=1}^\ell(z-z_k)\,\,\,\mbox{where} \,\,\mathcal 
H=\{z_1,\dots, z_\ell\}.
\end{equation}

\begin{proposition}\label{Proposition 3.3}
Let $g\in C^1(\overline\Omega), g\vert_{\mathcal
O_\epsilon}\equiv 0.$ Then for each $\delta\in (0,1)$ there exists a
constant $C(\delta)$ such that
\begin{equation}\label{(3.6)}
\Vert \widetilde R_{\Phi} (\overline {r(z)}g)\Vert_{L^2(\Omega)}
\le C(\delta)\Vert g\Vert_{C^1(\overline\Omega)}/\tau^{1-\delta},\quad
\Vert R_{\Phi} (r(z)g)\Vert_{L^2(\Omega)}\le C(\delta)\Vert
g\Vert_{C^1(\overline\Omega)}/\tau^{1-\delta}.
\end{equation}
\end{proposition}

\begin{proof}
Denote $v=\widetilde R_\Phi (\overline{r(z)}g).$ By Proposition
\ref{Proposition 3.2}
 \begin{equation}\label{(3.7)} \Vert
v\Vert_{L^2(\mathcal O_{\epsilon/2})}\le C/\tau.
\end{equation}
Then by Proposition \ref{Proposition 3.1}
$$
\frac{\partial v}{\partial z}+\tau \frac{\partial \Phi}{\partial
z}v=\overline{r(z)}g\quad \mbox{in}\,\,\Omega.
$$
There exists a function $p$ such that
$$
-\frac{\partial p}{\partial\overline z}+\tau
\frac{\overline{\partial\Phi(z)}}{\partial z}
p=v\quad\mbox{in}\,\,\Omega
$$
and there exists a constant $C>0$ independent of $\tau$ such that
\begin{equation}\label{(3.8)}
\Vert p\Vert_{L^2(\Omega)}\le C\Vert v\Vert_{L^2(\Omega)}.
\end{equation}
Let $\chi$ be a nonnegative  function such that $\chi\equiv 0$ on
$\mathcal
O_{\frac{\epsilon}{16}} $ and $\chi\equiv 1$ on  $\Omega\setminus
\mathcal O_{\frac \epsilon 8}.$ 
Setting $\widetilde p=\chi p$ and using $g\vert_{\mathcal O_\epsilon}
\equiv 0$, we have that
$$
\int_\Omega\overline{r( z)} g\overline p dx=\int_{\Omega\setminus
\mathcal
O_\epsilon}\overline{ r(z)} g\overline p dx=\int_\Omega\overline{r( z)}
g\overline{\widetilde p}dx
$$
and
\begin{equation}\label{(3.8')}
-\frac{\partial \widetilde p}{\partial\overline z}+\tau
\frac{\overline{\partial\Phi(z)}}{\partial z}\widetilde  p=\chi
v-p\frac{\partial\chi}{\partial \overline z}\quad\mbox{in}\,\,\Omega.
\end{equation}
 Then
\begin{equation}\label{(3.9)}\Vert \chi^\frac 12v\Vert^2_{L^2(\Omega)}
=\int_\Omega\overline{r(z)}
g\overline p dx+\int_\Omega p\frac{\partial\chi}{\partial \overline z}
\overline v dx.
\end{equation}

Note that
\begin{equation}\label{(3.10)}
\Vert \widetilde p\Vert_{H^1(\Omega)}\le C\tau\Vert p\Vert_{L^2(\Omega)}
\le
C\tau\Vert v\Vert_{L^2(\Omega)}, \quad
\int_\Omega\overline{r(z)} g\overline p dx=\int_\Omega
g\overline{r(z)\widetilde p}dx.
\end{equation}
Taking the scalar product of (\ref{(3.8')}) and $
\frac{\overline{r(z)}}{\partial_z\Phi(z)} g$ we obtain
$$
\int_\Omega\frac{\overline{r(z)}}{\partial_z\Phi(z)}
g\overline{\left(-\frac{\partial \widetilde p}{\partial\overline z}
+\tau
\frac{\overline{\partial\Phi(z)}}{\partial z}\widetilde
p\right)}dx
= \int_\Omega\frac{\overline{r( z)}}{\partial_z\Phi(z)}
g\overline{\left(\chi v-p\frac{\partial\chi}{\partial \overline z}
\right)}dx,
$$
$$
\tau\int_\Omega g\overline{r(z)\widetilde p}dx=\int_\Omega
\frac{\overline{r(z)}}{\partial_z\Phi(z)} g\overline{\left(\chi
v+p\frac{\partial\chi}{\partial \overline
z}\right)}dx
- \int_\Omega\frac{\partial}{\partial
z}\left(\frac{\overline{r(z)}}{\partial_z\Phi(z)} g\right)
\overline{\widetilde p}dx.
$$
By (\ref{(3.10)}) and the Sobolev embedding theorem, for each
$\epsilon\in (0,\frac 12)$ we have
\begin{eqnarray}
\left\vert\int_\Omega\frac{\partial}{\partial
z}\left(\frac{\overline{r(z)}}{\partial_z\Phi(z)} g\right)
\overline{\widetilde p} dx\right\vert \le
\left\vert\int_\Omega\frac{\overline{r(z)}\partial^2_z\Phi(z)}
{(\partial_z\Phi(z))^2}g\overline{\widetilde p}dx\right\vert
+ \left\vert\int_\Omega\frac{\overline{r(z)}}{\partial_z\Phi(z)}
\frac{\partial
g}{\partial z}\overline{\widetilde p}dx\right\vert \nonumber\\
\le C\Vert
g\Vert_{C^1(\overline\Omega)}
\left\Vert\frac{1}{\partial_z\Phi(z)}\right\Vert
_{L^{2-\epsilon}(\Omega)}\Vert\widetilde
p\Vert_{L^{\frac{2-\epsilon}{1-\epsilon}}(\Omega)}\le C\Vert \widetilde
p\Vert_{H^{\delta_3(\epsilon)}(\Omega)}\le C\tau^{\delta_4}\Vert
v\Vert_{L^2(\Omega)}.
\end{eqnarray}

Here we choose $\delta_3(\epsilon) > 0$ such that $\delta_3(\epsilon)
\rightarrow +0$ as $\epsilon\rightarrow +0$ and
$H^{\delta_3(\epsilon)}(\Omega)\subset
L^{\frac{2-\epsilon}{1-\epsilon}}(\Omega).$  Therefore
\begin{equation}\label{(3.11)}
\left\vert\int_\Omega gr(z)\overline{\widetilde p} dx\right\vert\le
C\tau^{-1+\delta_4}\Vert v\Vert_{L^2(\Omega)}\quad\mbox{as}\,\,
\delta_4\rightarrow +0.
\end{equation}
By (\ref{(3.7)})
\begin{equation}\label{(3.12)}
\left\vert\int_\Omega p\frac{\partial\chi}{\partial \overline z}
\overline vdx\right\vert
\le C\Vert p\Vert_{L^2(\Omega)}\Vert v\Vert_{L^2(\mathcal
O_\frac \epsilon 8)}\le C\Vert p\Vert_{L^2(\Omega)}/\tau.
\end{equation}
By (\ref{(3.8)}), (\ref{(3.11)}) and (\ref{(3.12)}) we obtain from
(\ref{(3.9)})
$$
\Vert v\Vert^2_{L^2(\Omega)}\le C(\tau^{-1+\delta_4}\Vert
v\Vert_{L^2(\Omega)}+\Vert p\Vert_{L^2(\Omega)}/\tau)\le
C\tau^{-1+\delta_4}\Vert v\Vert_{L^2(\Omega)}.
$$
In the last estimate we used (\ref{(3.8)}).
\end{proof}

We construct the function $u_{11}$ in the form
\begin{equation}\label{mamy1}
u_{11}=(u_{11,1}+u_{11,2}),
\end{equation}
where the functions $u_{11,k}$ are
defined in the following way: Let $e_i\in C^\infty(\overline
\Omega)$, $e_1+e_2\equiv 1,$  $e_2$ is zero in some neighborhood of
$\mathcal H$ and $e_1$ is zero in a neighborhood of
$\partial\Omega.$ The second term $u_{11}$ in the asymptotic
(3.7), is constructed  to satisfy
\begin{equation}\label{(3.15)}
\Delta u_{11}+4\tau \frac{{\partial\Phi(z)}}{\partial
z}\partial_{\overline z}u_{11}= aq_1+o\left(\frac 1\tau\right) \quad
\mbox{in}\,\,\Omega.
\end{equation}
Let $m_1(z), m_2(z), m_3(z)$ be polynomials satisfying
$$
(\partial^{-1}_{\overline z} (aq_1)-m_1(z))\vert_{\mathcal H}=0,
$$
$$
m_2(z)\vert_{\mathcal H}=0,\quad(\partial_z(\partial^{-1}_{\overline z}
(aq_1)-m_1(z))-m_2(z))\vert_{\mathcal H}=0,
$$
$$
m_3(z)\vert_{\mathcal H}=\partial_{z}m_3(z)\vert_{\mathcal
H}=0,\quad
\partial^2_z(\partial^{-1}_{\overline z} (aq_1)-m_1(z)-m_2(z)- m_3(z))
\vert_{\mathcal H}=0.
$$
The equation for $u_{11}$ can be transformed into
$$
4\partial_{z}u_{11}+4\tau \frac{{\partial\Phi(z)}}{\partial z}
u_{11}=\partial^{-1}_{\overline z} (aq_1)-\sum_{k=1}^3 m_k(z)
+ o\left(\frac 1\tau\right)\,\,\mbox{in}\,\,\Omega.
$$
Then
$$
4\partial_{z}u_{11,1}+4\tau \frac{{\partial\Phi(z)}}{\partial z}
u_{11,1}=e_1\left(\partial^{-1}_{\overline z} (aq_1)-\sum_{k=1}^3 m_k(z)
\right)\,\,\mbox{in}\,\,\Omega.
$$
and we define $u_{11,1}$ as
\begin{equation}\label{(3.16)}
u_{11,1}(x)=\frac 14 \widetilde R_\Phi\left(
e_1(\partial^{-1}_{\overline z}
(aq_1)-\sum_{k=1}^3 m_k(z))\right)
\end{equation}
and we define $u_{11,2}$ as
\begin{equation}\label{(3.17)}
u_{11,2}(x)=\frac 14
e_2(x)\left( \partial^{-1}_{\overline z} (aq_1)-\sum_{k=1}^3
m_k(z)\right)/(\tau
\partial_z\Phi(z)).
\end{equation}

Since by the assumption the function $e_2$ vanishes near the zeros of $\Phi$,
the function $u_{11,2}$ is smooth.

We will apply Proposition \ref{Proposition 3.3} to the function
$u_{11,1}$ to obtain the asymptotic behavior in $\tau.$  In order
to do that we need to represent the function
\begin{equation}\label{definition1}
\mathcal G_1=e_1\left(\partial^{-1}_{\overline z} (aq_1)
-\sum_{k=1}^3m_k(z)\right)
\end{equation}
in the form
$$
\mathcal G_1=\overline{r(z)}g(x),
$$
where $g$ is some function from
$C^1(\overline\Omega).$
This is an equivalent representation of the function
$m=\partial^{-1}_{\overline z} (aq_1)-\sum_{k=1}^3 m_k(z)$ in the form
$$
m=\overline{r(z)} g_1,\quad g_1\in C^1(\overline\Omega).
$$
We remind that the polynomial $r(z)$ is given by (\ref{polynomial}).
Denote as $p=\partial^{-1}_{\overline z} (aq_1).$ Let $x_j$ be a
critical point of the function
 $Im\Phi$ and $z_j \in \mathcal{H}$ (see (\ref{polynomial})).
By Taylor's formula
 $p(x)=p(z_j)+p_1(z-z_j)+p_2(\overline z-\overline z_j)+p_{11}(z-z_j)^2
+p_{12} (z-z_j)(\overline z-\overline z_j)
 +p_{22}(\overline z-\overline z_j)^2+q(z,\overline z).$
Then
$m=p_2(\overline z-\overline z_j)+p_{22}(\overline z-\overline z_j)^2
+p_{12} (z-z_j)(\overline z-\overline z_j)+q(z,\overline z)$ and we set
$g_1=(p_2(\overline z-\overline z_j)+p_{22}(\overline z-\overline z_j)^2
+p_{12} (z-z_j)(\overline z-\overline z_j)+q(z,\overline z))
/\overline{r(z)}.$
Let us show that $g_1\in C^1(\overline \Omega).$
Obviously $(p_2(\overline z-\overline z_j)+p_{22}(\overline z
-\overline z_j)^2+p_{12} (z-z_j)(\overline z-\overline z_j))/
\overline{r(z)}$ is a smooth function and $\widetilde q(z,\overline z)
=q(z,\overline z)/\overline{r(z)}$ is of $C^1$ outside of $z=0$.
Continue the function $\widetilde q$ by zero on $z=0.$ Since
$q=o(\vert z\vert^3)$ the partial derivatives of this function
vanishes at zero.

By Proposition \ref{Proposition 3.3}
\begin{equation} \label{(3.19)}\Vert
u_{11,1}\Vert_{L^2(\Omega)}\le C(\delta)/\tau^{1-\delta}\quad
\forall\delta\in(0,1).
\end{equation}

\subsection{Construction of $u_{12}$}

We will define $u_{12}$ as a  solution to the inhomogeneous problem
\begin{equation}\label{(3.18)}
\Delta (u_{12}e^{\tau\varphi_1}) +q_1u_{12}e^{\tau\varphi_1}=(q_1u_{11}+\Delta u_{11,2})e^{\tau \Phi} -
L_1\left(\chi_1e^{\tau\Phi(\frac{1}{\overline z})}
a\left(\frac{1}{\overline z}\right)\right)\quad \mbox{in}\,\,\Omega,
\end{equation}
\begin{equation}
u_{12}\vert_{\Gamma_-}=u_{11}e^{\tau\mbox{Im}\,\Phi}.
\end{equation}
 This can be done since
$$
\Vert q_1u_{11}+\Delta u_{11,2}\Vert_{L^2(\Omega)}\le
C(\delta)/\tau^{1-\delta}\quad \forall\delta\in(0,1)
$$
and by (\ref{(2.29)}), (\ref{kisa11})
$$
\left\Vert L_1\left(\chi_1e^{\tau\Phi(\frac{1}{\overline z})}
a\left(\frac{1}{\overline z}\right)\right)e^{-\tau\varphi_1}
\right\Vert_{L^2(\Omega)}=o\left(\frac{1}{\tau^2}\right).
$$
and  by (\ref{(3.4)}), (\ref{(3.16)}), (\ref{(3.17)})
$$
\Vert u_{11}\Vert_{C^0(\partial\Omega)}\le \frac{C}{\tau}.
$$
By Proposition \ref{Proposition 2.3} there exists a solution to
(\ref{(3.18)}) satisfying
\begin{equation}\label{(3.20)}
\Vert u_{12}\Vert_{L^2(\Omega)}\le C(\delta)/\tau^{\frac
34-\delta}, \quad\forall\delta\in(0,1).
\end{equation}

\subsection{Replacing $\Phi$ by $-\overline \Phi$}

Now we construct complex geometrical optics solutions for the potential 
$q_2$ satisfying the
conditions of the Theorem \ref{main} but with $\Phi$ replaced by
$-\overline \Phi$ and the solution vanishes on $\mathcal S$.

This is very similar to what we have already done.

Consider the Schr\"odinger equation
\begin{equation}\label{2.1II}
L_2v=\Delta v+q_2v=0 \quad \text{in}\,\, \Omega.
\end{equation}


We will construct solutions to (\ref{2.1II}) of the form
\begin{equation}\label{mozila}
v_1(x)=e^{-\tau\overline{\Phi(z)}}\overline{b(z)}
-\chi_1(x)e^{-\tau\overline{\Phi(\frac{1}{\overline z})}}
\overline{b\left(\frac{1}{\overline z}\right)}
+e^{-\tau\overline\Phi}v_{11}+e^{-\tau\varphi_1}v_{12},
\quad v_1\vert_{\mathcal S}=0.
\end{equation}

The construction of $v_1$ repeats the corresponding steps of the
construction of $u_1.$ In fact the only difference is that the
parameter $\tau$ is negative or in terms of the weight function we
use $-\varphi_1$ instead of $\varphi_1.$  We provide the details
for the sake of completeness.
The amplitude
$b(z)$  has the following properties:
$$
b\in C^2(\overline\Omega),\quad \frac{\partial b}{\partial \overline z}
\equiv 0, \quad b(z)\ne 0\,\,\mbox{on}\,\,\overline \Omega.
$$
Next we construct the cut-off function $\chi_2(x)$ with $\mbox{supp}\,
\chi_2\in \mathcal O_2$ where $\mathcal O_2$ is a neighborhood of
$\mathcal S = \partial\Omega \setminus \Gamma_{-,\epsilon}$, and
\begin{equation}\label{(2.291)}
\widetilde \varphi_1(x)>\varphi(x),\quad \,\, \forall x\in \Omega\cap
\mathcal O_2,
\end{equation}
\begin{equation}\label{(2.301)}
\partial \Omega\cap \mathcal O_2\subset
\partial\Omega_{+,\frac{\epsilon'}{2}},
\end{equation}
\begin{equation}\label{kisa2}
\mbox{supp}\,\nabla\chi_2\subset\subset B(0,1)\cap {\mathcal O}_2,
\end{equation}
\begin{equation}\label{ozon}
\mbox{supp}\chi_2\cap \mbox{supp}\chi_1=\emptyset.
\end{equation}
Consider the following integral
$$
\widetilde J(\tau)=\int_{\Omega}  \chi_2 r(x) e^{-\tau
\overline{\Phi(\frac{1}{\overline z})}+\tau{\Phi(z)}}dx.
$$
Similarly to Proposition \ref{miura} we have
\begin{proposition}\label{miura1}
Let $r\in C^{1+\ell}(\overline \Omega)$ for some positive
$\ell.$ Then
$$
\widetilde J(\tau)=o\left(\frac 1\tau\right).
$$
\end{proposition}

Now we construct $v_{11}.$  Let  $e_i\in C^\infty(\overline \Omega)$ ,
$e_1(x)+e_2(x)\equiv 1,$  $e_2$ is zero on
some neighborhood of $\mathcal H$ and $e_1$ is zero on some
neighborhood of $\partial\Omega$. Then
$$
\Delta v_{11}-4\tau\frac{\overline{\partial\Phi(z)}}{\partial z}
\partial_{ z}v_{11}=\overline bq_2+o\left(\frac 1\tau\right).
$$
Let $\widetilde m_1(\overline z),\widetilde m_2(\overline z),
\widetilde m_3(\overline z)$ be polynomials satisfying
$$
(\partial^{-1}_{ z} (\overline bq_2)-\widetilde m_1(\overline z))
\vert_{\mathcal H}=0,
$$
$$
\widetilde m_2(\overline z)\vert_{\mathcal H}=0,\quad
(\partial_{\overline
z}(\partial^{-1}_{ z} (\overline bq_2)-\widetilde m_1(\overline z))
-\widetilde m_2(\overline z))\vert_{\mathcal H}=0
$$
and
$$
\widetilde m_3(\overline z)\vert_{\mathcal H}=\partial_{\overline z}
\widetilde m_3(\overline z)\vert_{\mathcal H}=0,\quad
\partial^2_{\overline z}(\partial^{-1}_{ z} (\overline bq_2)
-\widetilde m_1(\overline z)-\widetilde m_2(\overline z)
-\widetilde m_3(\overline z))
\vert_{\mathcal H}=0.
$$
The equation for $v_{11}$ can be transformed into
$$
4\partial_{\overline z}v_{11}-4\tau
\frac{\overline{\partial\Phi(z)}}{\partial z}
v_{11}=\left(\partial^{-1}_{ z} (\overline bq_2)-\sum_{k=1}^3
\widetilde m_k(\overline z)\right) + o\left(\frac 1\tau\right).
$$
Let
\begin{equation}\label{mamy2}
v_{11}=v_{11,1}+v_{11,2}.
\end{equation}
Then
$$
4\partial_{\overline z}v_{11,1}-4\tau
\frac{\overline{\partial\Phi(z)}}{\partial
z}v_{11,1}
= e_1\left(\partial^{-1}_{ z} (\overline bq_2)-\sum_{k=1}^3\widetilde
m_k(\overline z)\right)\quad\mbox{in}\,\,\Omega,
$$
and we take $v_{11,1}$ as
\begin{equation}\label{(3.25)}
v_{11,1}=\frac 14 R_\Phi\left(e_1\left(\partial^{-1}_{ z} (\overline
bq_2)-\sum_{k=1}^3\widetilde m_k(\overline z)\right)\right)
\end{equation}
and we take $v_{11,2}$ as
\begin{equation}\label{(3.26)}
v_{11,2}=\frac 14 e_2(x)\left(\partial^{-1}_{
z} (\overline bq_2)-\sum_{k=1}^3\widetilde m_k(\overline z)\right)
/\left(\tau \overline{\frac{\partial\Phi}{\partial z}}\right).
\end{equation}

Thanks to our assumption on the function $e_2$,
this function is smooth.
Let us show that we can apply Proposition
\ref{Proposition 3.2} to the function
$v_{11,1}.$ In order to do that we need to represent the function
\begin{equation}\label{definition2}
\mathcal G_2=e_1\left(\partial^{-1}_{ z} (\overline
bq_2)-\sum_{k=1}^3 \widetilde m_k(\overline z)\right)
\end{equation}
in the form
$$
\mathcal G_2=zg(x),
$$
where $g$ is some function from $C^1(\overline\Omega).$
This is an equivalent representation of the function
$m=\partial^{-1}_{ z} (\overline bq_2)-\sum_{k=1}^3
\widetilde m_k(\overline z)$ in the form
$$
m={r(z)} g_1,\quad g_1\in C^1(\overline\Omega).
$$
Denote as $p=\partial^{-1}_{ z} (\overline bq_2).$ Let $x_j$ be a
critical point of the function $Im\Phi$ and $z_j$ be an arbitrary
critical point of the function $\Phi$. By Taylor's formula
 $p(x)=p(x_j)+p_1(x_j)(z-z_j)+p_2(x_j)(\overline z-\overline z_j)+p_{11}(z-z_j)^2+p_{12}
(z-z_j)(\overline z-\overline z_j)+p_{22}(\overline z-\overline
z_j)^2+q(z,\overline z).$
 Then
 $m=p_1(x_j)( z-z_j)+p_{11}(z-z_j)^2+p_{12} (z-z_j)
(\overline z-z_j)+q(z,\overline z)$ and we set $g_1=
 (p_1(x_j)( z-z_j)+p_{11}(z-z_j)^2+p_{12} (z-z_j)(\overline z-\overline z_j)
+q(z,\overline z))/{r(z)}.$
 Let us show that $g_1\in C^1(\overline \Omega).$ 
Obviously $(p_1( z- z_j)+p_{11}(z- z_j)^2+p_{12} (z-z_j)
(\overline z-\overline z_j))/ {r(z)}$ is a
smooth function and $\widetilde q(z,\overline z) =q(z,\overline
z)/{r(z)}$ is $C^1$ outside of $z=0$. Continue the function
$\widetilde q$ by zero on $z=0.$ Since $q=o(\vert z\vert^3)$ the
partial derivatives of this function vanishes at zero.

By Proposition \ref{Proposition 3.2}
\begin{equation} \label{(3.28)}
\Vert v_{11,2}\Vert_{L^2(\Omega)}+\Vert v_{11,1}\Vert_{L^2(\Omega)}
\le C(\delta)/\tau^{1-\delta},\quad \forall\delta\in(0,1).
\end{equation}

Let $v_{12}$ be a solution to the problem
\begin{equation}\label{(3.27)}
\Delta (v_{12}e^{-\tau\varphi_1})+q_2v_{12}e^{-\tau\varphi_1}=(q_2v_{11}+\Delta v_{11,2})e^{-\tau \Phi} -L_2\left(
\chi_2e^{-\tau\overline{\Phi(\frac{1}{\overline z})}}
\overline{b\left(\frac{1}{\overline z}\right)} \right) \quad
\mbox{in}\,\,\Omega,
\end{equation}
\begin{equation}
v_{12}\vert_{\mathcal S}=v_{11}e^{\tau\mbox{Im}\,\Phi}.
\end{equation}
Then since
$$
\Vert q_2v_{11}+\Delta v_{11,2}\Vert_{L^2(\Omega)}\le
C(\delta)/\tau^{1-\delta}, \quad \forall\delta\in(0,1)
$$
and by (\ref{kisa2})
$$
\left\Vert L_2\left(\chi_2
e^{-\tau\overline{\Phi(\frac{1}{\overline z})}}
\overline{b\left(\frac{1}{\overline z}\right)}\right)
e^{\tau\varphi_1}\right\Vert_{L^2(\Omega)}
= o\left(\frac{1}{\tau^2}\right),
$$
and by (\ref{(3.4)}), (\ref{(3.26)}), (\ref{(3.25)})
$$
\Vert v_{11}\Vert_{C^0(\partial\Omega)}\le \frac{C}{\tau},
$$
by Proposition \ref{Proposition 2.3} there exists a solution to
problem (\ref{(3.27)}) such that
\begin{equation}\label{(3.29)}
\Vert v_{12}\Vert_{L^2(\Omega)}\le C(\delta)/\tau^{\frac
34-\delta}, \quad\forall\delta\in(0,1).
\end{equation}
\section{\bf Proof of the theorem}

\begin{proposition}\label{Lemma 3.1}
Suppose that $\Phi$ satisfies (\ref{zzz}),(\ref{mika}), (\ref{mmm})
and (\ref{mmm1}). Let $\{x_1,\dots,x_\ell\}$ be the set of critical
points of the function $Im\Phi$. Then for any potentials $q_1,q_2\in
C^\ell(\bar\Omega),\ell>1$ with the same DN maps and for any holomorphic
functions $a$ and $b$, we have
$$
\sum_{k=1}^\ell\frac{(qa\overline b)(x_k)}{\vert(\mbox{det}
\thinspace \mbox{Im}\Phi'')(x_k)\vert^\frac 12}=0,\quad q=q_1-q_2.
$$
\end{proposition}

 \begin{proof}
Let $u_1$ be a solution to (\ref{2.1I}) and satisfy (\ref{mozila}),
and $u_2$ be a solution to the following equation
$$
\Delta u_2+q_2u_2=0\quad \mbox{in}\,\,\Omega,\quad
u_2\vert_{\partial \Omega}=u_1,\quad \nabla
u_2\vert_{\Gamma_{-,\epsilon}}=\nabla u_1.
$$
Denoting $u=u_1-u_2$ we obtain
\begin{equation}\label{(3.22)}
\Delta u+q_2u=-qu_1\quad \mbox{in}\,\,\Omega,\quad u\vert_{\partial
\Omega}=\frac{\partial u}{\partial
\nu}\vert_{\Gamma_{-,\epsilon}}=0.
\end{equation}

We multiply (\ref{(3.22)}) by $v$ and
integrate over $\Omega.$  By (\ref{(3.20)}) and (\ref{(3.29)}),
we have
\begin{eqnarray}\label{(3.31)}
0=\int_\Omega qu_1vdx=\int_\Omega q(a\overline b+\overline
bu_{11}+av_{11})e^{\tau(\Phi
(z)-\overline{\Phi(z)})}dx\nonumber\\
+ \int_\Omega \Biggl( q \chi_1(x)e^{\tau\Phi(\frac{1}{\overline z})}
a\left(\frac{1}{\overline z}\right)\overline b
e^{-\tau\overline{\Phi(z)}}
+ q \chi_1(x)e^{-\tau\overline{\Phi(\frac{1}{\overline z})}}
\overline{b\left(\frac{1}{\overline z}\right)}ae^{\tau
\Phi(z)}\nonumber\\
+ q \chi_1(x)e^{\tau\Phi(\frac{1}{\overline z})}
a\left(\frac{1}{\overline z}\right)\chi_2(x)
e^{-\tau\overline{\Phi(\frac{1}{\overline z})}}
\overline{b\left(\frac{1}{\overline z}\right)}\Biggr)dx
+ o\left(\frac 1\tau\right).
\end{eqnarray}
By Propositions \ref{miura} and \ref{miura1}
$$
\int_\Omega \left( q \chi_1(x)e^{\tau\Phi(\frac{1}{\overline z})}
a\left(\frac{1}{\overline z}\right)\overline b
e^{-\tau\overline{\Phi(z)}}+q \chi_2(x)e^{-\tau\overline{\Phi(\frac{1}{\overline z})}}\overline{b\left(\frac{1}{\overline z}\right)}ae^{\tau
\Phi(z)}\right)dx = o\left(\frac{1}{\tau}\right).
$$
By (\ref{ozon})
$$
\int_\Omega q \chi_1(x)e^{\tau\Phi(\frac{1}{\overline z})}
a\left(\frac{1}{\overline z}\right)\chi_2(x)e^{-\tau\overline
{\Phi(\frac{1}{\overline z})}}\overline{b\left(\frac{1}{\overline z}\right)}dx
=0.
$$
Therefore we can rewrite (\ref{(3.31)}) as
\begin{equation}\label{mooo}
\sum_{k=1}^\ell\frac{\pi(qa\overline b)(x_k)
e^{2i\tau\mbox{Im}\thinspace\Phi(x_k)}}
{\tau \vert(\mbox{det}\thinspace \mbox{Im}\Phi'')(x_k)\vert^\frac 12}
+ \int_\Omega q(\overline
bu_{11}+av_{11})e^{\tau(\Phi(z)-\overline{\Phi(z)})}dx
+ o\left(\frac 1\tau\right)=0.
\end{equation}
By (\ref{(3.17)}), (\ref{(3.26)}), (\ref{mamy1}), (\ref{mamy2}) and the fact
that
\begin{eqnarray}
\int_\Omega  \overline{b}q u_{11,2}e^{\tau(\Phi(z)-\overline{\Phi(z)})}dx=
\nonumber \\
\frac {1}{4\tau}\int_\Omega  \overline{b}q  \frac{e_2(\partial^{-1}_{\overline
z} (aq_1)-\sum_{k=1}^3m_k(z))}{\tau
\partial_z\Phi(z)}e^{\tau(\Phi(z)-\overline{\Phi(z)})}dx
=o\left(\frac 1\tau\right),
\end{eqnarray}
and the fact that
\begin{eqnarray}
\int_\Omega aq v_{11,2}e^{\tau(\Phi(z)-\overline{\Phi(z)})}dx=\nonumber \\\frac
{1}{4\tau}\int_\Omega  aq \frac{e_2(\partial^{-1}_{ z} (\overline
bq_2)-\sum_{k=1}^3\widetilde m_k(\overline z))}{\tau
\overline{\partial_z\Phi(z)}}
e^{\tau(\Phi(z)-\overline{\Phi(z)})}dx
= o\left(\frac 1\tau\right),
\end{eqnarray}
which follows from the stationary phase argument and $e_2\vert
_{\mathcal H}=0$, we obtain
\begin{equation}\label{mooo1}
 \sum_{k=1}^\ell\frac{\pi(qa\overline b)(x_k)
e^{2i\tau\mbox{Im}\thinspace\Phi(x_k)}}
{\tau \vert(\mbox{det}\thinspace \mbox{Im}\Phi'')(x_k)\vert
^\frac 12}
+\int_\Omega q(\overline
bu_{11,1}+av_{11,1})e^{\tau(\Phi(z)-\overline{\Phi(z)})}dx
+ o\left(\frac 1\tau\right)=0.
\end{equation}
By (\ref{(3.1)}), (\ref{(3.26)}) and (\ref{(3.16)})
\begin{eqnarray}
 0=\sum_{k=1}^\ell\frac{\pi(qa\overline b)(x_k)
e^{2i\tau\mbox{Im}\thinspace\Phi(x_k)}}
{\tau \vert(\mbox{det}\thinspace \mbox{Im}\Phi'')(x_k)\vert
^\frac 12}
-  \frac 14\int_\Omega q(
\overline b\widetilde R_\Phi \mathcal G_1+a R_\Phi\mathcal
G_2)e^{\tau(\Phi(z)-\overline{\Phi(z)})}dx
+ o\left(\frac 1\tau\right)=\nonumber\\
\sum_{k=1}^\ell\frac{\pi(qa\overline b)(x_k)
e^{2i\tau\mbox{Im}\thinspace\Phi(x_k)}}
{\tau \vert(\mbox{det}\thinspace
\mbox{Im}\Phi'')(x_k)\vert^\frac 12}-\frac 14\int_\Omega( (\partial_z^{-1}
(q\overline b))\mathcal
G_1+(\partial_{\overline z}^{-1}(qa))\mathcal G_2)
e^{\tau (\Phi(z)-\overline{\Phi(z)})}dx
+ o\left(\frac 1\tau\right)=\nonumber\\
\sum_{k=1}^\ell\frac{\pi(qa\overline b)(x_k)
e^{2i\tau\mbox{Im}\thinspace\Phi(x_k)}}
{\tau \vert(\mbox{det}\thinspace
\mbox{Im}\Phi'')(x_k)\vert^\frac 12}+o\left(\frac 1\tau\right).
\end{eqnarray}
We remind the definitions of the functions $\mathcal G_1$ and
$\mathcal G_2$ introduced in (\ref{definition1}) and
(\ref{definition2}).

In order to get rid of
the integral $\int_\Omega( (\partial_z^{-1}(q\overline b))\mathcal
G_1+(\partial_{\overline z}^{-1}(qa))\mathcal
G_2)e^{\tau(\Phi(z)-\overline{\Phi(z)})}dx$, we used the stationary
phase lemma  (see e.g. Theorem 7.7.5 \cite{H}) and the fact that
$\mathcal G_1\vert_{\mathcal H}=\mathcal G_2\vert_{\mathcal H}=0.$
Passing to the limit in this equality as $\tau\rightarrow +\infty$
we obtain
$$
\lim_{\tau\to\infty} \sum_{k=1}^\ell\frac{\pi(qa\overline b)(x_k)
e^{2i\tau\mbox{Im}\thinspace\Phi(x_k)}}
{\vert(\mbox{det} \thinspace \mbox{Im}\Phi'')(x_k)\vert^\frac 12}
=0.
$$
The function $K(\tau) \equiv \sum_{k=1}^\ell
\frac{2\pi(qa\overline b)(x_k)
e^{2i\tau\mbox{Im}\thinspace\Phi(x_k)}}
{\vert(\mbox{det} \thinspace \mbox{Im}\Phi'')(x_k)
\vert^\frac 12}$ is almost periodic.  Therefore by the Bohr theorem
(e.g., \cite{BS}, p.493), we see that $K(\tau) = 0$ for all
$\tau \in \R$.  Thus setting $\tau=0$, we complete the
proof.
\end{proof}

Proposition \ref{Lemma 3.1} plays the key role in the proof of
 Theorem 1.1. In order to be able to use this proposition we need
to prove the existence of the weight function $\Phi$. The following
proposition will allow us to construct this function.

Let $\mathcal P_\epsilon$ be  a non-empty open subset of the
boundary $\partial\Omega$: the union of the segment between
$\widehat x_+$ and $\widehat x_{+,\epsilon}$ and the
segment between $\widehat x_{-,\epsilon}$ and $\widehat x_{-}.$

Consider the Cauchy problem for the Laplace operator
\begin{equation}\label{(4.112)}
\Delta\psi=0\quad\mbox{in}\,\,\Omega, \quad
\left(\psi,\frac{\partial\psi}{\partial\nu}\right)\vert
_{\partial\Omega\setminus
\mathcal P_\epsilon} =(a,b).
\end{equation}

The following proposition establishes the solvability of (\ref{(4.112)})
for a dense set of Cauchy data.
\begin{proposition}
There exist a set $\mathcal O\subset C^2(\overline{\partial\Omega
\setminus\mathcal P_\epsilon})\times C^1(\overline{\partial\Omega
\setminus\mathcal P_\epsilon})$ such that for each $(a,b)\in \mathcal O$, 
problem (\ref{(4.112)}) has at least one solution $\psi\in C^2(\overline\Omega)$and $\overline{\mathcal O}$ $=C^2(\overline{\partial\Omega\setminus 
\mathcal P_\epsilon})$$\times C^1(\overline{\partial\Omega\setminus
\mathcal P_\epsilon}).$
\end{proposition}

\begin{proof} First we observe that without the loss of generality  we may
assume that $a\equiv 0.$
Consider the following extremal problem
\begin{equation}\label{extr}
J(\psi)=\left\Vert \frac{\partial \psi}{\partial \nu}
-b\right\Vert^2
_{H^2(\partial\Omega\setminus\mathcal P_\epsilon)}
+ \epsilon \Vert \psi\Vert_{H^2(\partial\Omega)}^2
+ \frac 1\epsilon\left\Vert\Delta^2\psi\right\Vert^2_{L^2(\Omega)}
\rightarrow \inf,
\end{equation}
\begin{equation}\label{extr1}
\psi\in \mathcal X.
\end{equation}
Here $\mathcal X=\left\{\delta(x)\vert \delta\in H^2(\Omega),
\Delta^2 \delta\in L^2(\Omega), \Delta\delta\vert_{\partial\Omega}
=\delta\vert_{\partial\Omega\setminus\mathcal P_\epsilon}=0,
\delta\vert_{\partial\Omega}\in H^2(\partial\Omega), \frac{\partial
\delta}{\partial \nu}\in H^2(\partial\Omega\setminus\mathcal P_\epsilon)\right\}$.

For each $\epsilon>0$ there exists a unique solution to
(\ref{extr}) and (\ref{extr1}), which we denote as 
$\widehat \psi_\epsilon$.
By the Fermat theorem (see e.g., \cite{AL} p. 155) we have
$$
J'(\widehat \psi_\epsilon)[\delta]=0,\quad \forall\delta\in\mathcal X.
$$

This equality can be written in the form
$$
\left(\frac{\partial \widehat\psi_\epsilon}{\partial \nu}-b,
\frac{\partial \delta}{\partial \nu} \right)
_{H^2(\partial\Omega\setminus\mathcal P_\epsilon)}+\epsilon
( \widehat\psi_\epsilon,\delta)_{H^2(\partial\Omega)}+\frac 1\epsilon
(\Delta^2\widehat\psi_\epsilon,\Delta^2\delta)_{L^2(\Omega)}=0.
$$
This equality implies that the sequence
$\{\frac{\partial \widehat\psi_\epsilon}{\partial \nu}\}$ is bounded
in $H^2(\partial\Omega\setminus\mathcal P_\epsilon)$, the sequence
$\{{\epsilon} \widehat\psi_\epsilon\}$ converges to zero in $H^2(\partial\Omega)$ and $\left\{\frac {1}{{\epsilon}}
\Delta^2\widehat\psi_\epsilon\right\}$ is bounded in $L^2(\Omega).$

Therefore there exist $q\in H^2(\partial\Omega\setminus\mathcal 
P_\epsilon)$ and  $p\in L^2(\Omega)$ such that
\begin{equation}\label{zopa}
\frac{\partial \widehat\psi_{\epsilon_k}}{\partial \nu}-b
\rightharpoonup q \quad \mbox{ weakly in} \,\, H^2(\partial\Omega\setminus\mathcal P_\epsilon)
\end{equation}
and
\begin{equation}\label{zima}
\left(q,\frac{\partial \delta}{\partial \nu} \right)
_{H^2(\partial\Omega\setminus\mathcal P_\epsilon)}
+ (p,\Delta^2\delta)_{L^2(\Omega)}=0\quad \forall \delta\in\mathcal X.
\end{equation}

Next we claim that
 \begin{equation}\label{Laplace}
 \Delta p=0\quad\mbox{in}\,\,\Omega
 \end{equation}
in the sense of distributions. Suppose that (\ref{Laplace}) is 
already proved.  This implies
 $$
 (p,\Delta^2\delta)_{L^2(\Omega)}=0\quad \forall \delta\in H^4(\Omega), \quad\Delta\delta\vert_{\partial\Omega}=\frac{\partial \Delta\delta}{\partial\nu}\vert_{\partial\Omega}=0.
 $$
 This equality and (\ref{zima}) imply that
 \begin{equation}\label{zima1}
\left(q,\frac{\partial \delta}{\partial \nu} \right)
_{H^2(\partial\Omega\setminus\mathcal P_\epsilon)}=0\quad \forall \delta\in H^4(\Omega), \Delta\delta\vert_{\partial\Omega}=\frac{\partial \Delta\delta}{\partial\nu}\vert_{\partial\Omega}=0.
\end{equation}
Then using the trace theorem, we conclude that $q=0$
and (\ref{zopa}) implies that
$$
\frac{\partial \widehat\psi_{\epsilon_k}}{\partial \nu}
-b\rightharpoonup 0\quad\mbox{weakly in}\,\, {H^2(\partial\Omega\setminus
\mathcal P_\epsilon)}.
$$
By the Sobolev embedding theorem
$$
\frac{\partial \widehat\psi_{\epsilon_k}}{\partial \nu}-b\rightarrow 0\quad\mbox{in}\,\, C^1(\partial\Omega\setminus\mathcal P_\epsilon).$$
Therefore the sequence
$
\{\widehat \psi_{\epsilon_k}-\widetilde \psi_{\epsilon_k}\},
$
with $$\Delta\widetilde \psi_{\epsilon_k}=
\Delta\widehat \psi_{\epsilon_k}\quad\mbox{in}\,\,\Omega,\,\,\,\widetilde \psi_{\epsilon_k}\vert_{\partial\Omega}=0
$$
represents the desired approximation for solution of the Cauchy
problem (\ref{(4.112)}).

Now we prove (\ref{Laplace}). Let $\widetilde x$ be an arbitrary point
in $\Omega$ and let $\widetilde \chi$ be a smooth function such that
it is zero in some neighborhood of $\partial\Omega\setminus
\mathcal P_\epsilon$ and the set
$\mathcal B=\{x\in \Omega\vert \widetilde \chi(x)=1\}$ contains an
open connected subset $\mathcal F$ such that $\widetilde x\in
\mathcal F$
and  $ \mathcal P_\epsilon\cap \overline{\mathcal F}$ is an open set in $\partial\Omega.$  By (\ref{zima})
$$
0=(p,\Delta^2(\widetilde\chi\delta))_{L^2(\Omega)}
=(\widetilde \chi p,\Delta^2 \delta)_{L^2(\Omega)}
+(p,[\Delta^2,\widetilde\chi]\delta)_{L^2(\Omega)}.
$$
That is,
\begin{equation}\label{vorona}
(\widetilde \chi p,\Delta^2\delta)_{L^2(\Omega)}+([\Delta^2,\widetilde\chi]^*p,\delta)_{L^2(\Omega)}=0\quad \forall \delta\in \mathcal X.
\end{equation}
This equality implies that $\widetilde\chi p\in H^1(\Omega).$

Next we take another smooth cut off function $\widetilde\chi_1$ such that $\mbox{supp}\,\widetilde\chi_1\subset\mathcal B.$ A neighborhood of $\widetilde x$ belongs to  $\mathcal B_1=\{x\vert \widetilde\chi_1=1\}$, the interior of
$\mathcal B_1$ is connected, and $\mbox{ Int }\mathcal B_1\cap
\mathcal P_\epsilon$ contains an open subset $\mathcal O$ in
$\partial\Omega.$  Similarly to (\ref{vorona}) we have
\begin{equation}\label{vorona}
(\widetilde \chi_1 p,\Delta^2\delta)_{L^2(\Omega)}+([\Delta^2,\widetilde\chi_1]^*p,\delta)_{L^2(\Omega)}=0.
\end{equation}

This equality implies that $\widetilde\chi_1 p\in H^2(\Omega).$
Let $\omega$ be a domain such that $\omega\cap\Omega=\emptyset$, $\partial\omega\cap\partial\Omega\subset\mathcal O$ contains a set open in $\partial\Omega.$

We extend $p$ on $\omega$ by zero. Then
$$
(\Delta(\widetilde \chi_1 p),\Delta \delta)_{L^2(\Omega\cup\omega)}+([\Delta^2,\widetilde\chi]^*p,\delta)_{L^2(\Omega\cup\omega)}=0.
$$
Hence
$$
\Delta^2 (\widetilde\chi_1 p)=0 \quad\mbox{in} \,\,\mbox{ Int }\mathcal B_1\cup\omega,\quad p\vert_\omega=0.
$$
By the Holmgren theorem $\Delta(\widetilde\chi_1 p)\vert
_{\mbox {Int }\mathcal B_1}=0$, that is, $(\Delta p)(\widetilde x)=0.$
\end{proof}

\noindent{\bf Completion of the proof of Theorem 1.1.}
It suffices to prove that
$q(0)=0.$ We take $\mathcal P_\epsilon$ in the previous proposition
to be the union of the segment between $\widehat x_+$ and $\widehat
x_{+,\epsilon}$ and the segment between $\widehat x_{-,\epsilon}$
and $\widehat x_{-}.$

We will show that $q_1(0)= q_2(0).$ By obvious changes of the argument below we
can prove that $q_1( x)=q_2(x)$ for any point $x\in \Omega.$

Suppose that $\psi(x)$ is a solution to (\ref{(4.112)}) for some Cauchy data.
Next, since $\Omega$ is simply connected,  we construct a function $\varphi$ such that  the function
$\Phi(z)=\varphi(x)+i\psi(x)$ is holomorphic in $\Omega$. Consider the function
$\widetilde \Phi(z)=z^2\Phi(z)$. Observe that $Im\widetilde
\Phi=(x_1^2-x_2^2)\psi(x)+2x_1x_2\varphi(x). $ In particular by
(\ref{(4.112)}) and the Cauchy-Riemann equations, we have
$$
Im\widetilde \Phi\vert_{\partial\Omega\setminus \mathcal P_\epsilon}=
(x_1^2-x_2^2)a(x)+2x_1x_2c(x), \quad \frac{\partial
c(x)}{\partial\tau}=b(x).
$$
Since we can choose $a,b$ from a dense set in
$C^1(\overline{\partial\Omega\setminus \mathcal P_\epsilon})$ and
the tangential derivatives of $(x_1^2-x_2^2)$ and $x_1x_2$ are not
equal zero simultaneously, we can choose
 $a,b$ such that
\begin{equation} \label{(4.200)}\frac{\partial \mbox{Im}\widetilde \Phi}{\partial\tau}\vert_{\overline \Gamma_-}=\frac{\partial \mbox{Re}\widetilde \Phi}{\partial\nu}
\vert_{\overline \Gamma_-}<0, \quad\frac{\partial \mbox{Im}\widetilde
\Phi}{\partial\tau}\vert_{\overline{
\partial\Omega\setminus \Gamma_{-,\epsilon}}}=\frac{\partial \mbox{Re}\widetilde \Phi}{\partial\nu}
\vert_{\overline{
\partial\Omega\setminus \Gamma_{-,\epsilon}}}>0.\end{equation}
Obviously the function $\widetilde \Phi$ has a critical point at zero.
We may assume that $\partial^2_z \widetilde \Phi(0)\ne 0$. Really if
$\Phi(0) \ne 0$ then $\partial^2_z\widetilde \Phi(0)=2\Phi(0)$. If
$\Phi(0)=0$ we modify this function by adding a small real number:
$\Phi(z)+\tilde\epsilon.$ Obviously we will have (\ref{(4.200)}).

A general function $\widetilde\Phi$ may have degenerate critical
points. In order to avoid them,
we approximate the function $\widetilde \Phi$ in $C^1(\overline \Omega)$
by a sequence of holomorphic functions $\{\widetilde\Phi_k\}_{k=1}
^\infty$ such that
\begin{equation}\label{(4.3)}
\widetilde\Phi_k\rightarrow \widetilde\Phi\quad \mbox{in}
\,\,C^1(\overline\Omega), \quad \frac{\partial \mbox{Re}
\widetilde\Phi_k}{\partial\nu} \vert_{\overline \Gamma_-}<0,\quad
\frac{\partial \mbox{Re} \widetilde\Phi_k}{\partial\nu}
\vert_{\overline{
\partial\Omega\setminus \Gamma_{-,\epsilon}}}>0,
\end{equation}
\begin{equation}\label{(4.4)}
\mathcal H_k=\{z\vert\partial_{z}\Phi_k(z)=0\},\,\,card \mathcal
H_k<\infty,\,\,\mathcal\, H_k\cap \partial\Omega=\{\emptyset\},\,\,
\partial^2_{z} \widetilde\Phi_k(z_\ell)\ne 0,\,\, \forall z_\ell\in
\mathcal H_k.
\end{equation}

Let us show that such a sequence exists. For any $\epsilon_1\in (0,1)$ we
consider a function $\widetilde \Phi(z/(1+\epsilon_1))$. Obviously
$$
\widetilde \Phi(\cdot/(1+\epsilon_1))\rightarrow \widetilde\Phi\quad\mbox{in}\,\,C^1(\overline \Omega), \quad \mbox{as}\,\, \epsilon_1\rightarrow +0.
$$
Each function $ \widetilde \Phi(z/(1+\epsilon_1))$ is holomorphic in $B(0,1+\epsilon_1)$ and in $B(0,1)$ it can be approximated by a polynomial.
Let $\epsilon_1\in (0,1)$ be an arbitrary but fixed. Consider the sequence of such polynomials. Let $p(z)=\sum_{k=0}^\kappa c_k z^k$ be a polynomial from this sequence.
Consider the polynomial $p'(z)=\sum_{k=1}^\kappa kc_k z^{k-1}=\Pi_{k=1}^{\ell}(z-\widehat z_k)^{s(k)}.$
Here we assume $\widehat z_j\ne \widehat z_k$ for $k\ne j.$
Let us construct an approximation of the polynomial $p(z)$ by a sequence of polynomials of the order $\kappa.$ We do the construction in the following way. First pick up all $s(k)$ such that $s(k)\ge 2.$ Denote the set of such
indices as $\mathcal U.$ Let $\widehat k\in \mathcal U.$
Consider the sequences $\{\widehat z_{k,\ell_1, \epsilon_2}\},
\dots, \{\widehat z_{k,\ell_{s(\widehat k)},\epsilon_2}\}$ such that
$$
\widehat z_{k,\ell_j,\epsilon_2}\rightarrow \widehat z_k
\quad \mbox{as}\,\,\epsilon_2\rightarrow +0,\,\,\forall \ell_j
\in\{\ell_1,\dots, \ell_{s(\widehat k)}\},
$$
$$
\widehat z_{k,\ell_j,\epsilon_2}\ne
\widehat z_{k,\ell_{j_1},\epsilon_2},
\quad 1\le k \le \kappa,\,\, \mbox{if}\,\, \ell_j\ne \ell_{j_1}.
$$
The polynomial
$$
p_{\epsilon_2}'(z)=\Pi_{k=1}^{\ell}\Pi_{j=1}^{s(k)}(z-\widehat
z_{k,j,\epsilon_2})
$$
does not have any zeros of order greater then one.
By the construction we have
$$
p_{\epsilon_2}'(z)=\sum_{k=1}^\kappa kc_{k,\epsilon_2} z^{k-1}
$$ satisfying
$$
c_{k,\epsilon_2}\rightarrow c_k,\quad \forall k\in\{1,\dots,\kappa\}.
$$
This means that the sequence of polynomials
$p_{\epsilon_2}(z)=\sum_{k=0}^\kappa c_{k,\epsilon_2} z^{k}, 
c_{0,\epsilon_2}
=c_0$ converges to $p(z)$ in $C^1(\Omega)$ and for small $\epsilon_2$ these 
polynomials do not have critical points.

Let us fix some sufficiently large  $\widehat k$ and consider $k>\widehat
k.$ Then $card \,\mathcal H_{k_1}=card \,\mathcal H_{k_2}$ for all
$k_1>\widehat k$ and $k_2>\widehat k.$ Let $card\, \mathcal H_k=\ell$ and
points $z_1=\widetilde x_{1,1}+i\widetilde x_{2,1},\dots, z_\ell=\widetilde
x_{1,\ell}+i\widetilde x_{2,\ell}$ represent all critical points of the
function $\widetilde\Phi_k(z)=\varphi_k(z)+i\psi_k(z). $

Thanks to (\ref{(4.3)}) and (\ref{(4.4)}), we can apply Proposition 
\ref{Lemma 3.1}.  We have
$$
\sum_{j=1}^\ell \frac{q(\widetilde x_j)}{\vert \mbox{det}
\psi_k''(\widetilde x_j)\vert^\frac 12}=0,\quad
\widetilde x_j=(\widetilde x_{1,j},\widetilde x_{2,j}).
$$
Let $\widehat j\in \{1,\dots,\ell\}$ be an arbitrary number. Consider the polynomial
$$
p(z)= \frac{d_1}{2}\frac{\Pi^\ell_{k\ne {\widehat
j}}(z-z_k)^3} {\Pi^\ell_{k\ne {\widehat j}}(z_{\widehat j}-z_k)^3}
(z-z_{\widehat
j})^2+d\frac{\Pi^\ell_{k\ne {\widehat j}}(z-z_k)^3} 
{\Pi^\ell_{k\ne {\widehat
j}}(z_{\widehat j}-z_k)^3}(z-z_{\widehat j}).
$$
Then
\begin{eqnarray}\label{(4.5)}
 \partial^2_zp(z_{\widehat j})=d_1\in \Bbb C,\,\,
\partial_zp(z_{\widehat j})=d\in \Bbb C,\\ p(z_j)=\partial_zp(z_j)=
\partial^2_zp(z_j)=0\,\, j\in \{1,\dots,\ell\}\setminus\{\widehat j\}.\nonumber
\end{eqnarray}

Consider the function $\widetilde\Phi_k(z)+\varepsilon p(z).$ 
By (\ref{(4.5)}) for small
$\varepsilon$ the set of critical points of this function consists
exactly of
$\ell$ points,  which we denote as  $z_j(\varepsilon)$
($\widetilde x_j(\varepsilon)=
(\mbox{Re}z_j(\varepsilon),\mbox{Im}z_j(\varepsilon))$). These critical
points have the following properties:
\begin{equation}\label{(4.6)}
z_j(0)=z_j,\quad \frac{\partial z_j(\varepsilon)}{\partial \varepsilon}\vert
_{\varepsilon=0}=0, \quad j\ne \widehat j, \quad
\frac{\partial z_{\widehat j}(\varepsilon)}{\partial \varepsilon}\vert
_{\varepsilon=0}=-\frac{d}{\partial^2_z\widetilde\Phi_k(z_{\widehat j})}.
\end{equation}
In fact, there exists $\varepsilon_0>0$ such that
$$
z_j=z_j(\varepsilon),\quad \forall \varepsilon\in
(-\varepsilon_0,\varepsilon_0), \quad j\ne \widehat j.
$$
Then by Proposition \ref{Lemma 3.1} we have
$$
J(\varepsilon)=\sum_{j=1}^\ell \frac{q(\widetilde x_j(\varepsilon))}{\vert
\mbox{det} (\psi_k+\varepsilon \mbox{Im}\,p)''(\widetilde  x_j(\varepsilon))\vert^\frac
12}=0.
$$
Taking the derivative of the function $J(\varepsilon)$ at zero, we have:
\begin{eqnarray} \label{(4.7)}
-\frac{1}{\vert {\partial^2_z\widetilde\Phi_k(z_{\widehat j})}\vert^2} 
\frac{\partial_{x_1}q(\widetilde x_{\widehat j}(0))
\mbox{Re}(d\overline{\partial^2_z\widetilde\Phi_k(z_{\widehat j})})
+ \partial_{x_2}q(\widetilde x_{\widehat
j}(0))\mbox{Im}(d\overline{\partial^2_z\widetilde\Phi_k(z_{\widehat
j})})}{\vert \mbox{det} \psi_k''
(\widetilde x_{\widehat j}(0))\vert^\frac 12}\\
+\frac 12\sum_{j=1}^\ell( \frac{q(\widetilde
x_j(0))(2\partial^2_{x_1x_2}\psi_{k}(\widetilde x_j(0))\mbox{Im}\, \partial^2_{x_1x_2}p(\widetilde
x_j(0))+2\partial^2_{x_1x_1}\psi_{k}(\widetilde x_j(0))\mbox{Im}\,\partial^2_{x_1x_1}p(\widetilde
x_j(0))}{\vert \mbox{det} \psi_k''(\widetilde
x_j(0))\vert^\frac 32}\nonumber\\
+\frac 12 q(\widetilde x_j(0))\frac{\partial_{x_1}(\mbox{det}
\psi_k''(\widetilde x_j(0))
 \mbox{Re}(d\overline{\partial^2_z\widetilde\Phi_k(z_{\widehat j})})
 +\partial_{x_2}(\mbox{det} \psi_k''(\widetilde x_{\widehat j}(0))
 \mbox{Im}(d\overline{\partial^2_z\widetilde\Phi_k(z_{\widehat j}))}}
 {{\vert \overline{\partial^2_z\widetilde\Phi_k(z_{\widehat j})}\vert^2}\vert
 \mbox{det} \psi_k''(\widetilde x_{\widehat j}(0))\vert^\frac 32})=0.\nonumber
\end{eqnarray}

The first and third terms of (\ref{(4.7)}) are independent
 of $ \mbox{Im}\,\partial^2_{x_1x_2}p(\widetilde x_j(0))$ and $\mbox{Im}\,\partial^2_{x_1x_1}p(\widetilde x_j(0))$.  Consequently
$$
\frac 12\sum_{j=1}^\ell \frac{q(\widetilde
x_j(0))(-2\partial_{x_1x_2}^2\psi_{k}(\widetilde x_j(0))\mbox{Im}\,\partial_{x_1x_1}^2p(\widetilde
x_j(0))-2\partial_{x_1x_1}^2\psi_{k}(x_j(0)) \mbox{Im}\,\partial_{x_1x_1}^2 p(\widetilde
x_j(0))}{\vert \mbox{det} \psi_k''(\widetilde x_j(0))\vert^\frac 32}=0.
$$
This formula and (\ref{(4.6)}) imply that $q(\widetilde x_{\widehat
j}(0))=0.$ Since by (\ref{(4.3)}) and (\ref{(4.4)}) the set
$\mathcal H_k$ converges to the set of critical points of
$\widetilde \Phi$ and $0$ belongs to the set of critical points of
$\widetilde \Phi$, we have $q(0)=0.$ $\blacksquare$

%

\end{document}